\documentclass[reqno]{amsart}
\usepackage{amssymb,eucal,latexsym, mathrsfs,xy,enumerate,graphicx,url}

\xyoption{all}

\newenvironment{enumeratei}{\begin{enumerate}[\upshape (i)]}%
{\end{enumerate}}
{\end{enumerate}}
\newenvironment{enumerater}{\begin{enumerate}[\upshape (1)]}%
{\end{enumerate}}

\usepackage[usenames]{color}

\newcommand{\pup}[1]{\textup{(}{#1}\textup{)}}

\newcommand{\msem}{meet-semi\-group}

\newcommand{\ext}[2]{{#1}\langle{#2}\rangle}

\newcommand{\eqdef}{\underset{\mathrm{def}}{=}}

\newcommand{\Var}{\operatorname{{\mathbf{Var}}}}

\newcommand{\Matn}{\operatorname{Mat}}
\newcommand{\Mat}[1]{\Matn({#1})}
\newcommand{\Matp}[2]{\operatorname{M}^{{\oplus}}_{{#1}}\left({#2}\right)}
\newcommand{\Radn}{\operatorname{Rad}}
\newcommand{\Rad}[2]{\Radn_{{#1}}\left({#2}\right)}
\newcommand{\Rd}[1]{\operatorname{Rad}\left({#1}\right)}
\newcommand{\Wrn}{\operatorname{Wr}}
\newcommand{\Wr}[2]{\Wrn_{{#1}}\left({#2}\right)}

\newcommand{\Bis}{\mathbf{Bis}}
\newcommand{\pz}[1]{{#1}^{\sqcup0}}

\newcommand{\ip}[1]{\lfloor{#1}\rfloor}
\newcommand{\sdf}{\mathbin{\text{\rotatebox[origin=c]{90}{$\oslash$}}}}
\newcommand{\spl}{\mathbin{\triangledown}}

\DeclareMathOperator{\card}{card}
\DeclareMathOperator{\Idp}{Idp}

\DeclareMathOperator{\NSub}{NSub}
\DeclareMathOperator{\Int}{Int}
\DeclareMathOperator{\Typ}{Typ}
\DeclareMathOperator{\typ}{typ}

\newcommand{\ga}{\alpha}
\newcommand{\gb}{\beta}

\newcommand{\gf}{\varphi}

\newcommand{\gs}{\sigma}

\newcommand{\gO}{\Omega}
\newcommand{\gS}{\Sigma}

\newcommand{\fS}{\mathfrak{S}}

\newcommand{\sd}{\mathbin{\smallsetminus}}

\newcommand{\lep}{\leq^{+}}

\newcommand{\one}{\mathbf{1}}

\newcommand{\ol}[1]{\overline{#1}}

\newcommand{\pI}[1]{\bigl({#1}\bigr)}
\newcommand{\pII}[1]{\Bigl({#1}\Bigr)}
\newcommand{\pIII}[1]{\biggl({#1}\biggr)}

\newcommand{\set}[1]{\left\{#1\right\}}
\newcommand{\setm}[2]{\set{{#1}\mid{#2}}}
\newcommand{\vecm}[2]{\left({#1}\mid{#2}\right)}

\newcommand{\dd}{\mathbf{d}}
\newcommand{\rr}{\mathbf{r}}

\newcommand{\id}{\mathrm{id}}

\newcommand{\es}{\varnothing}

\newcommand{\Fbn}{\operatorname{F_{bis}}}
\newcommand{\Fb}[1]{\Fbn(#1)}
\newcommand{\Fin}{\operatorname{F_{inv}}}
\newcommand{\Fi}[1]{\Fin(#1)}
\newcommand{\Ubn}{\operatorname{U_{bis}}}
\newcommand{\Ub}[1]{\Ubn(#1)}

\newcommand{\bLb}{\boldsymbol{\Lambda}_{\mathrm{Bis}}}
\newcommand{\bLg}{\boldsymbol{\Lambda}_{\mathrm{Grp}}}
\newcommand{\tbLg}{\widetilde{\boldsymbol{\Lambda}}_{\mathrm{Grp}}}

\newcommand{\NN}{\mathbb{N}}
\newcommand{\ZZ}{\mathbb{Z}}

\DeclareMathOperator{\Con}{Con}

\DeclareMathOperator{\dom}{dom}

\newcommand{\fI}{\mathfrak{I}}

\newcommand{\cB}{{\mathcal{B}}}
\newcommand{\cC}{{\mathcal{C}}}

\newcommand{\cF}{{\mathcal{F}}}
\newcommand{\cG}{{\mathcal{G}}}
\newcommand{\cH}{{\mathcal{H}}}
\newcommand{\cI}{{\mathcal{I}}}

\newcommand{\cK}{{\mathcal{K}}}

\newcommand{\cV}{{\mathcal{V}}}

\numberwithin{equation}{section}

\newtheorem*{stat}{\name}
\newcommand{\name}{testing}

\newenvironment{all}[1]{\renewcommand{\name}{#1}\begin{stat}}
                        {\end{stat}}

\theoremstyle{plain}

\newtheorem{theorem}{Theorem}[section]
\newtheorem{proposition}[theorem]{Proposition}
\newtheorem{corollary}[theorem]{Corollary}
\newtheorem{lemma}[theorem]{Lemma}
\newtheorem{examplepf}[theorem]{Example}

\theoremstyle{definition}

\newtheorem{definition}[theorem]{Definition}
\newtheorem{notation}[theorem]{Notation}

\theoremstyle{remark}

\newcommand{\qedc}{{\qed}~{\rm Claim~{\theclaim}.}}
\newcommand{\qedsc}{{\qed}~{\rm Claim.}}

\numberwithin{figure}{section}
\numberwithin{table}{section}

\newcommand{\ba}{\boldsymbol{a}}
\newcommand{\bb}{\boldsymbol{b}}
\newcommand{\bc}{\boldsymbol{c}}

\newcommand{\be}{\boldsymbol{e}}

\newcommand{\bx}{\boldsymbol{x}}

\newcommand{\bga}{\boldsymbol{\alpha}}

\newcommand{\bgf}{\boldsymbol{\varphi}}
\newcommand{\bgy}{\boldsymbol{\psi}}

\newcommand{\gq}{\theta}
\newcommand{\bgq}{\boldsymbol{\theta}}

\newcommand{\bI}{\boldsymbol{I}}

\newcommand{\va}{\mathsf{a}}
\newcommand{\vb}{\mathsf{b}}

\newcommand{\vp}{\mathsf{p}}
\newcommand{\vq}{\mathsf{q}}
\newcommand{\vr}{\mathsf{r}}

\newcommand{\vs}{\mathsf{s}}
\newcommand{\vx}{\mathsf{x}}

\newcommand{\vu}{\mathsf{u}}

\newcommand{\scL}{\mathbin{\mathscr{L}}}
\newcommand{\scR}{\mathbin{\mathscr{R}}}
\newcommand{\scD}{\mathbin{\mathscr{D}}}

\title{Varieties of Boolean inverse semigroups}

\author[F. Wehrung]{Friedrich Wehrung}
\address{LMNO, CNRS UMR 6139\\
D\'epartement de Math\'ematiques\\
Universit\'e de Caen Normandie\\
14032 Caen cedex\\
France}
\email{friedrich.wehrung01@unicaen.fr}
\urladdr{http://www.math.unicaen.fr/\~{}wehrung}

\date{\today}

\subjclass[2010]{20M18; 08B10; 08B15; 06F05; 08A30; 08A55; 08B05; 08B20; 20E22; 20M14}

\keywords{Semigroup; monoid; inverse; Boolean; bias; variety; group; wreath product; additive homomorphism; conical; refinement monoid; index; type monoid; generalized rook matrix; fully group-matricial; radical; congruence; residually finite}

\begin{document}

\begin{abstract}
In an earlier work, the author observed that Boolean inverse semigroups, with semigroup homomorphisms preserving finite orthogonal joins, form a congruence-permutable variety of algebras, called \emph{biases}.
We give a full description of varieties of biases in terms of varieties of groups:
\begin{enumerater}
\item Every free bias is residually finite.
In particular, the word problem for free biases is decidable.

\item Every proper variety of biases contains a largest finite symmetric inverse semigroup, and it is generated by its members that are generalized rook matrices over groups with zero.

\item There is an order-preserving, one-to-one correspondence between proper varieties of biases and certain finite sequences of varieties of groups, descending in a strong sense defined in terms of wreath products by finite symmetric groups.
\end{enumerater}
\end{abstract}

\maketitle

\section{Introduction}\label{S:Intro}
\emph{Boolean inverse semigroups} are an abstraction of semigroups of partial transformations which are closed under finite disjoint unions, originally studied by Tarski's school (cf.~\cite{Tars49}).
These objects have been for the last decade an active topic of research, see, for example, \cite{KLLR15,Laws10,Laws12,LaLe13,LawSco14}.
By definition, an inverse semigroup~$S$ with zero is Boolean, if its semilattice of idempotents is (generalized) Boolean and~$S$ has finite orthogonal joins.
(We refer to Section~\ref{S:Basic} for precise definitions.)

Unlike classes of structures such as groups, inverse semigroups, modules, rings, Lie algebras, the class of Boolean inverse semigroups is not defined as a variety of algebras in the sense of universal algebra: while the multiplication and the inversion of an inverse semigroup are \emph{full} operations, orthogonal join is only a \emph{partial} operation.
The author introduced in~\cite{WBIS} two full operations~$\sdf$ and~$\spl$ (cf.~\eqref{Eq:Defsdf} and~\eqref{Eq:Defspl} for precise definitions), defined, on every Boolean inverse semigroup, in terms of multiplication, inversion, and the partial operation of orthogonal join, such that the semigroup homomorphisms preserving~$\sdf$ and~$\spl$ are exactly the \emph{additive semigroup homomorphisms};
by definition, a semigroup homomorphism is additive if it preserves all finite orthogonal joins.
Moreover, Boolean inverse semigroups can be characterized \emph{via} a finite system of identities in the similarity type $(0,{}^{-1},\cdot,\sdf,\spl)$.
The models of those identities are called \emph{biases}.
Hence, the category of biases, with bias homomorphisms, is identical to the category of Boolean inverse semigroups, with additive semigroup homomorphisms.
We also prove in~\cite{WBIS} that the variety of biases is \emph{congruence-permutable}, which makes Boolean inverse semigroups much closer, in spirit, to groups and rings than to semigroups.

The question was thus raised, during the open problem session of the 2016 Workshop on New Directions in Inverse Semigroups in Ottawa, whether there could be a convenient way to describe the \emph{varieties of biases}, that is, the solution classes, within biases, of sets of identities in the similarity type of biases.
In this paper we solve that question, by proving that any variety~$\cV$ of biases, distinct from the variety of all biases (we say that~$\cV$ is \emph{proper}), is determined by the (necessarily finite) descending sequence of group varieties~$\cG_n$, where~$\cG_n$ (which we will call the \emph{$n$-th radical} of~$\cV$, see Notation~\ref{Not:nRad}) is the variety of all groups~$G$ such that the bias~$\Matp{n}{\pz{G}}$ of all generalized rook matrices of order~$n$ over the Boolean inverse semigroup $\pz{G}=G\cup\set{0}$ (we call such structures \emph{groups with zero}) belongs to~$\cV$.
Moreover, we describe, in terms of wreath products by finite symmetric groups, which descending sequences of varieties of groups arise in this way, thus yielding, in Theorem~\ref{T:VarOrdBias}, an isomorphism between proper varieties of biases and certain finite descending sequences of group varieties.
In particular, the lattice of all proper varieties of biases embeds, as a sublattice, into a countable power of the lattice of all group varieties enlarged by the empty class.

Let us now summarize, section by section, the organization of the paper.

In \textbf{Section~\ref{S:Basic}}, we recall the main definitions and facts needed through the paper.

In \textbf{Section~\ref{S:FreeBias}}, we analyze the structure of the universal bias of an arbitrary inverse semigroup, proving in particular that finiteness of the latter implies finiteness of the former.
We deduce the following result:

\begin{all}{Theorem~\ref{T:BISdecid}}
Every free bias is residually finite.
In particular, the word problem for free biases is decidable.
\end{all}

\textbf{Section~\ref{S:Rook}} is devoted to introducing a few elementary tools on generalized rook matrices and the type monoid.
This yields, as a byproduct, the following result:

\begin{all}{Proposition~\ref{P:BISFinIdp}}
The Boolean inverse monoids with finite sets of idempotents are exactly the finite products of monoids of the form~$\Matp{n}{\pz{G}}$.
\end{all}

In \textbf{Section~\ref{S:VarBIS}} we prove that the nonzero elements, of the type monoid of a (finitely) subdirectly irreducible bias, form a downward directed subset.

In \textbf{Section~\ref{S:GenVarBias}} we prove the following result:

\begin{all}{Theorem~\ref{T:VarMnG}}
Every proper variety of biases is generated by its members which are biases of generalized rook matrices over groups with zero, with an upper bound on the order of those matrices.
\end{all}

In \textbf{Section~\ref{S:GRMgp}} we observe that the invertible elements of a bias of generalized rook matrices, over a group with zero, can be described as a wreath product with a finite symmetric group.
This enables us to express, in group-theoretical terms, embeddability questions between biases of generalized rook matrices over groups with zero.

In \textbf{Section~\ref{S:Proj}} we prove that finite symmetric biases are projective, and we state a projectivity property of biases of generalized rook matrices, over groups with zero, within the class of all biases with cancellative type monoid.

In \textbf{Section~\ref{S:BdedInd}} we express the monoid-theoretical concept of index, studied for type monoids of biases in Section~\ref{S:GenVarBias}, in terms of a certain inverse semigroup identity, thus expressing, in the language of inverse semigroups, the concept of index of a variety of biases.

In \textbf{Section~\ref{S:VarOrder}} we prove our main technical result, Lemma~\ref{L:VarOrdbias}, which gives an exact description of the variety order on biases of generalized rook matrices over groups with zero.
This enables us to prove our main result:

\begin{all}{Theorem~\ref{T:VarOrdBias}}
There is a one-to-one, order-preserving correspondence between proper varieties of biases and finite, decreasing sequences $\vecm{\cG_k}{1\leq k\leq n}$ of varieties of groups, such that for all positive integers~$k$, $l$ with $kl\leq n$ and every $G\in\cG_{kl}$, the wreath product $G\wr\fS_k$ belongs to~$\cG_l$.
\end{all}

The sequence $\vecm{\cG_k}{1\leq k\leq n}$ associated to a variety~$\cV$ of biases is given as follows: $n$ is the largest nonnegative integer such that~$\cI_n\in\cV$ (the \emph{index} of~$\cV$, see Definition~\ref{D:Index}) and whenever $1\leq k\leq n$, the \emph{$k$-th radical}~$\cG_k$ is the variety of all groups~$G$ such that the bias~$\Matp{k}{\pz{G}}$, of all generalized rook matrices of order~$k$ over the bias~$\pz{G}$, belongs to~$\cV$.

\section{Basic concepts}\label{S:Basic}

An \emph{inverse semigroup} (cf.~\cite{Howie,Laws98}) is a semigroup~$S$ where every $x\in S$ has a unique inverse, that is, an element~$x^{-1}$ such that $x=xx^{-1}x$ and $x^{-1}=x^{-1}xx^{-1}$.
Every group, or every semilattice, is an inverse semigroup.

We shall denote by~$\Idp S$ the set of all idempotent elements in a semigroup~$S$.
For every element~$x$ in an inverse semigroup~$S$, the elements $\dd(x)=x^{-1}x$ and $\rr(x)=xx^{-1}$ are both idempotent.
The \emph{natural ordering} between elements~$x$ and~$y$ of~$S$, simply denoted by $x\leq y$, can be defined, among others, by any of the equivalent statements $x=y\dd(x)$ and $x=\rr(x)y$.
Recall that \emph{Green's relations}~$\scL$, $\scR$, and~$\scD$ can be defined on~$S$ by
 \begin{align*}
 x\scL y&\quad\text{if}\quad
 \dd(x)=\dd(y)\,;\\
 x\scR y&\quad\text{if}\quad
 \rr(x)=\rr(y)\,,
 \end{align*}
$\scD=\scL\circ\scR=\scR\circ\scL$ (cf. \cite[Proposition~II.1.3]{Howie}).
The relation~$\scD$ takes a particularly convenient form on the idempotent elements: namely, for all $a,b\in\Idp S$, the relation $a\scD b$ holds if{f} there exists $x\in S$ such that $a=\dd(x)$ and $b=\rr(x)$.

For a semigroup~$S$, we shall denote by~$\pz{S}$ the semigroup obtained by adding to~$S$ a new zero element~$0$ (i.e., $0\cdot x=x\cdot0=0$ for every~$x$).
In particular, if~$S$ is an inverse semigroup, then so is~$\pz{S}$.

Elements~$x$ and~$y$ in an inverse semigroup with zero are \emph{orthogonal}, in notation $x\perp y$, if $x^{-1}y=xy^{-1}=0$.
An inverse semigroup~$S$ with zero is \emph{Boolean} if $\Idp S$ is a generalized Boolean algebra and any two orthogonal elements~$x$ and~$y$ in~$S$ have a join with respect to the natural ordering, then denoted by $x\oplus y$.

For any elements~$x$ and~$y$ in a Boolean inverse semigroup~$S$ such that the meet $x\wedge y$ exists, we denote by~$x\sd y$ the unique element such that $x=(x\wedge y)\oplus(x\sd y)$.
Observe, in particular, that $x\sd y$ is always defined if~$x$ and~$y$ are compatible (i.e., $x^{-1}y$ and $xy^{-1}$ are both idempotent).
This encompasses the case where~$x$ and~$y$ are both idempotent, and also the one where~$x$ and~$y$ are comparable (i.e., $x\leq y$ or $y\leq x$).

Important examples of Boolean inverse semigroups are the finite symmetric inverse semigroups~$\fI_n$, for nonnegative integers~$n$, consisting of all partial one-to-one functions on the set $[n]=\set{1,2,\dots,n}$ under composition.

Every commutative monoid~$M$ can be endowed with a partial preordering~$\lep$, defined by
 \[
 x\lep y\qquad\text{if }x+z=y\text{ for some }z\in M\,.
 \]
We say that~$M$ is
\begin{itemize}
\item[---] \emph{conical} if $x+y=0$ implies that $x=y=0$, for all $x,y\in M$;

\item[---] a \emph{refinement monoid} if for all $a_0,a_1,b_0,b_1\in M$ such that $a_0+a_1=b_0+b_1$, there are elements $c_{i,j}\in M$, for $i,j\in\set{0,1}$, such that each $a_i=c_{i,0}+c_{i,1}$ and each $b_i=c_{0,i}+c_{1,i}$.

\end{itemize}
A partially ordered abelian group $(G,+,0,\leq)$ is a \emph{dimension group} if it is directed (as a poset), unperforated (i.e., $0\leq mx$ implies that $0\leq x$, whenever~$m$ is a positive integer and $x\in G$), and the positive cone $G^+\eqdef\setm{x\in G}{0\leq x}$ is a refinement monoid.

The commutative monoids we shall be mainly concerned with are the type monoids $\Typ S$, for Boolean inverse semigroups~$S$.
By definition, $\Typ S$ is the universal monoid of the partial monoid~$\Int S$ of all $\scD$-classes of elements of~$S$ (which we call the \emph{type interval} of~$S$), endowed with the partial addition defined by
 \[
 x/{\scD}+y/{\scD}=(x\oplus y)/{\scD}\,,
 \quad\text{whenever }x,y\in S\text{ are orthogonal}\,.
 \]
Moreover, as in~\cite{WBIS}, we shall write~$\typ_S(x)$, or sometimes simply~$\typ(x)$, instead of~$x/{\scD}$.
Since the canonical map from~$\Int S$ to~$\Typ S$ is one-to-one, $\typ(x)=\typ(y)$ if{f} $x\scD y$, for all $x,y\in S$.
By \cite[Corollary~4-1.4]{WBIS}, $\Int S$ is a lower interval of~$\Typ S$, generating~$\Typ S$ as a monoid, and~$\Typ S$ is a conical refinement monoid.

For a positive integer~$n$ and a Boolean inverse semigroup~$S$, a matrix $x=(x_{i,j})_{(i,j)\in[n]\times[n]}$ with entries in~$S$ is a \emph{generalized rook matrix of order~$n$} (cf. \cite[\S~4.5]{Wallis2013}, \cite{KLLR15}, also \cite[Section~3.5]{WBIS}) if the equalities $x_{i,j}^{-1}x_{i,k}=x_{j,i}x_{k,i}^{-1}=0$ hold whenever $i,j,k\in[n]$ with $j\neq k$.
The generalized rook matrices of order~$n$ over a Boolean inverse semigroup~$S$ form a Boolean inverse semigroup, denoted as in~\cite{WBIS} by~$\Matp{n}{S}$.
As in~\cite{WBIS}, we denote by $x_{(i,j)}$ the generalized rook matrix with $(i,j)$th entry equal to~$x$ and all other entries equal to zero, for every $x\in S$ and every $(i,j)\in[n]\times[n]$.

The following easy result is contained in \cite[Proposition 3-5.3]{WBIS}.

\begin{proposition}\label{P:IdpMnS}
Let~$S$ be a Boolean inverse semigroup and let~$n$ be a positive integer.
Then the idempotent elements of~$\Matp{n}{S}$ are exactly the diagonal matrices with idempotent entries.
\end{proposition}

\begin{definition}\label{D:Dcanc}
A Boolean inverse semigroup~$S$ is \emph{$\scD$-cancellative} if the conjunction of $a\oplus b=a'\oplus b'$ and $a\scD a'$ implies that $b\scD b'$, for all $a,b,a',b'\in\Idp S$ such that $a\perp b$ and $a'\perp b'$.
\end{definition}

Recall that an inverse monoid is \emph{factorizable} if for every $x\in S$ there is an invertible element $g\in S$ such that $x\leq g$.

\begin{proposition}\label{P:Dcanc}
A Boolean inverse semigroup~$S$ is $\scD$-cancellative if{f} its type monoid~$\Typ S$ is cancellative.
Furthermore, if~$S$ is unital, then this is equivalent to~$S$ be factorizable.
\end{proposition}

\begin{proof}
Let $\Typ S$ be cancellative and let $a,b,a',b'\in\Idp S$ such that $a\oplus b\scD a'\oplus b'$ and $a\scD a'$.
Since $\typ(a)+\typ(b)=\typ(a\oplus b)=\typ(a'\oplus b')=\typ(a')+\typ(b')$ and $\typ(a)=\typ(a')$, it follows from the cancellativity of~$\Typ S$ that $\typ(b)=\typ(b')$, that is, $b\scD b'$.

Suppose, conversely, that~$S$ is $\scD$-cancellative.
We claim that the type interval~$\Int S$ is cancellative.
Let $\ba,\bb,\bb'\in\Int S$ such that $\ba+\bb=\ba+\bb'$ within~$\Int S$.
By \cite[Lemma 4-1.6]{WBIS}, there are $a,a'\in\ba$, $b\in\bb$, and $b'\in\bb'$ such that $a\oplus b=a'\oplus b'$.
Since~$S$ is~$\scD$-cancellative, it follows that $b\scD b'$, that is, $\bb=\bb'$, which completes the proof of our claim.
Now $\Int S$ is a generating lower interval of the conical refinement monoid~$\Typ S$ (cf. \cite[Corollary 4-1.4]{WBIS}), thus, by our claim together with \cite[Corollary 2-7.4]{WBIS}, $\Typ S$ is cancellative.

Now let~$S$ be unital.
Suppose first that~$S$ is $\scD$-cancellative and let $x\in S$.
Set $u=1\sd\dd(x)$ and $v=1\sd\rr(x)$.
Then $1=\dd(x)\oplus u=\rr(x)\oplus v$ with $\dd(x)\scD\rr(x)$, thus, since~$S$ is $\scD$-cancellative, $u\scD v$, that is, there exists $y\in S$ such that $\dd(y)=u$ and $\rr(y)=v$.
The element $g=x\oplus y$ is invertible and $x\leq g$, thus completing the proof that~$S$ is factorizable.

Suppose, conversely, that~$S$ is factorizable and let $a,b,a',b'\in\Idp S$ such that $a\oplus b=a'\oplus b'$ and $a\scD a'$.
We must prove that $b\scD b'$.
Setting $c=1\sd(a\oplus b)$, we get $(a\oplus c)\oplus b=(a'\oplus c)\oplus b'=1$ with $a\oplus c\scD a'\oplus c$, thus reducing the problem to the case where $a\oplus b=a'\oplus b'=1$.
Let $x\in S$ such that $a=\dd(x)$ and $a'=\rr(x)$.
Since~$S$ is factorizable, there is an invertible element $g\in S$ such that $x\leq g$.
{}From $a'=gag^{-1}$ and the invertibility of~$g$ it follows that $b'=gbg^{-1}$, whence, since~$g$ is invertible, $b\scD b'$.
\end{proof}

Every Boolean inverse semigroup~$S$ can be endowed with the \emph{skew difference}~$\sdf$ and the \emph{skew addition}~$\spl$, respectively defined by
\begin{align}
x\sdf y&=\pI{\rr(x)\sd\rr(y)}x\pI{\dd(x)\sd\dd(y)}\,,
\label{Eq:Defsdf}\\
x\spl y&=(x\sdf y)\oplus y\,,\label{Eq:Defspl}
\end{align}
for all $x,y\in S$.
We prove in~\cite{WBIS} that the structures $(S,0,{}^{-1},\cdot,\sdf,\spl)$ can then be axiomatized by a finite number of identities, whose models are called \emph{biases}.
We prove in~\cite{WBIS} that for any Boolean inverse semigroups~$S$ and~$T$, a homomorphism $f\colon S\to T$ of semigroups with zero is a bias homomorphism if{f} it is \emph{additive}, that is, $f(x\oplus y)=f(x)\oplus f(y)$ whenever~$x$ and~$y$ are orthogonal elements in~$S$.
(In particular, $f(0)=0$.)
A nonempty subset~$I$ of~$S$ is an \emph{additive ideal} of~$S$ if $IS\cup SI\subseteq I$ and~$I$ is closed under finite orthogonal joins.
In that case, the inclusion map from~$I$ into~$S$ is an additive semigroup embedding, and~$I$ is a sub-bias of~$S$.

The bias congruences of a Boolean inverse semigroup~$S$ are characterized, in~\cite{WBIS}, as those inverse semigroup congruences~$\bgq$ such that for all $x\in S$ and all orthogonal idempotents~$a$ and~$b$ of~$S$, $xa\equiv_{\bgq}a$ and $xb\equiv_{\bgq}b$ implies that $x(a\oplus b)\equiv_{\bgq}a\oplus b$.
(Here and elsewhere, $x\equiv_{\bgq}y$ is an equivalent notation for $(x,y)\in\bgq$.)
We denote by~$\Con S$ the (algebraic) lattice of all bias congruences of any Boolean inverse semigroup~$S$.
For a bias congruence~$\bgq$ of a Boolean inverse semigroup~$S$, we shall usually denote by $\gq\colon S\twoheadrightarrow S/{\bgq}$ the canonical projection.

A \emph{similarity type} (cf.~\cite{MMTa}) is a pair $\gS=(\cF,\nu)$, where~$\cF$ is a set and~$\nu$ is a map from~$\cF$ to the nonnegative integers.
The elements of~$\cF$ should be thought of as \emph{function symbols} and~$\nu(f)$ should be thought of as the \emph{arity} of~$f$.
For example, the similarity type of groups is usually given by $\cF=\set{\cdot,{}^{-1}}$, $\nu(\cdot)=2$, and $\nu({}^{-1})=1$.
The similarity type of biases is given by $\cF=\set{0,{}^{-1},\cdot,\sdf,\spl}$, $\nu(0)=0$, $\nu({}^{-1})=1$, and $\nu(\cdot)=\nu(\sdf)=\nu(\spl)=2$.

In general, formal compositions of elements of~$\cF$, taking the arities into account, are called the \emph{terms} of~$\gS$.
An \emph{identity} of~$\gS$ is an expression of the form $\vp=\vq$, where~$\vp$ and~$\vq$ are both terms.
A \emph{$\gS$-algebra} is a nonempty set~$A$, endowed with a map which to each $f\in\cF$, with arity~$n$, associates a map $f^A\colon A^n\to A$ (just an element of~
$A$ if $n=0$).

A \emph{variety of $\gS$-algebras} is the class of all $\gS$-algebras that satisfy a given set of identities of~$\gS$.
Varieties, also known under the name of \emph{equational classes}, are defined and studied in any textbook of universal algebra such as \cite{BurSan,GrUA,MMTa}.
A standard reference for varieties of groups is Neumann's monograph~\cite{Neum1967}.
Every variety~$\cV$, on a set~$X$ of variables, is determined by the set of all identities, with set of variables~$X$, satisfied by~$\cV$.
This set of identities is, in turn, a \emph{fully invariant congruence} of the algebra of all terms on~$\gS$.
This correspondence gives an order-reversing bijection between varieties and fully invariant congruences of the term algebra (cf. \cite[Corollary~II.14.10]{BurSan}), and thus it enables us to dispose conveniently of the apparent foundational problem raised by varieties being proper classes.
In particular, the lattice of all varieties of $\gS$-algebras can be defined, and it has cardinality at most $2^{\aleph_0+\card\cF}$.
Moreover, the fully invariant congruences of an algebra~$A$ form a complete sublattice of the congruences of~$A$ (cf. \cite[Exercise~II.14.1]{BurSan}), thus the lattice of all subvarieties of a variety~$\cV$ satisfies the dual of every lattice identity satisfied by the congruence lattices of all members of~$\cV$.

Now the variety of all groups, and the variety of all biases, are both congruence-permutable (see \cite[Section~II.5]{BurSan} and \cite[Section~3-4]{WBIS}, respectively).
Since the congruence lattice of every congruence-permutable algebra satisfies the \emph{modular identity}, and in fact the even stronger \emph{Arguesian identity} (cf.~\cite{Jons53} and \cite[Theorem~410]{LTF}), and since the Arguesian identity is self-dual (cf. \cite{Jons1972}), the lattice~$\bLg$ of all varieties of groups and the lattice~$\bLb$ of all varieties of biases are both Arguesian.
Stronger congruence identities, following from congruence-permutability, were discovered by Mark Haiman in~\cite{Haim1985}.
For more on identities satisfied by normal subgroup lattices of groups or congruence lattices in algebras from congruence-permutable varieties, we refer the reader to \cite{BuOW1994,Freese1995}.

On the cardinality side, it is known since Ol$'$\v{s}anski\u{\i} that there are continuum many varieties of groups~\cite{Olsh1970}.

We denote by $\Var(\cC)$ the variety of groups generated by a class~$\cC$ of groups, and set $\Var(G)\eqdef\Var(\set{G})$ for any group~$G$.

For a poset~$P$, we denote by $P\sqcup\set{\infty}$ the poset obtained by adding to~$P$ a new top element.

We set $\ZZ^+=\set{0,1,2,3,\dots}$ and $\NN=\set{1,2,3,\dots}$.

\section{Free biases are residually finite}
\label{S:FreeBias}

In this section we prove that the satisfaction of any equation, in the universal bias of an inverse semigroup~$S$, can be reduced to a positive quantifier-free formula over~$S$ in the similarity type of inverse semigroups (Proposition~\ref{P:BIS2IS}).
We deduce from this that every free bias is residually finite, so, in particular, the word problem for finite biases is decidable (Theorem~\ref{T:BISdecid}).

The proof of the following lemma is an elementary calculation and we omit it.

\begin{lemma}\label{Eq:A-BleqC-D}
Let~$a$, $b$, $c$, $d$ be elements in a Boolean ring~$B$.
Then $(a\sd b)\wedge(c\sd d)=(a\wedge c)\sd(b\vee d)$.
Furthermore, $a\sd b\leq c\sd d$ if{f} $a\leq b\vee c$ and $a\wedge d\leq b$.
\end{lemma}

Since every inverse semigroup~$S$ has a semigroup embedding into a bias~$T$ such that $0_T\notin S$ (use the Vagner-Preston Theorem), the canonical map from~$S$ to its \emph{universal bias}~$\Ub{S}$ is a semigroup embedding with $0\notin S$, and we shall thus assume that this embedding is an inclusion map.
Therefore, $S$ is an inverse subsemigroup of~$\Ub{S}$, generating~$\Ub{S}$ as a bias, and every semigroup homomorphism from~$S$ to a bias~$T$ extends to a (unique) bias homomorphism from~$\Ub{S}$ to~$T$.

\begin{lemma}\label{L:CanFormBias}
Let~$S$ be an inverse subsemigroup of a Boolean inverse semigroup~$T$.
Suppose that~$T$ is generated by~$S$ as a bias.
Denote by~$B$ the
Boolean subring of~$\Idp{T}$ generated by~$\Idp{S}$.
Then $B=\Idp{T}$ and every element $\bx\in T$ can be written in the form
 \begin{equation}\label{Eq:CanFormBias}
 \bx=\bigoplus_{i=1}^nx_ia_i\,,\text{ where }n\in\ZZ^+\,,\ 
 a_1,\dots,a_n\in B\,,\text{ and }x_1,\dots,x_n\in S\,.
 \end{equation}
\end{lemma}

\begin{proof}
The set of all elements $a\in\Idp{T}$, such that $xax^{-1}\in B$ whenever $x\in S$, contains~$\Idp{S}$ and is closed under finite meets, differences, and orthogonal joins, thus it contains~$B$.
Hence, $xBx^{-1}\subseteq B$ whenever $x\in S$.
Set $\Delta=\setm{xa}{(x,a)\in S\times B}$.
For each $(x,a)\in S\times B$, $ax=axx^{-1}x=xx^{-1}ax=xa'$ where $a'=x^{-1}ax\in B$.
This also proves that~$\Delta$ is closed under the inversion operation $x\mapsto x^{-1}$.
Now for all $x,y\in S$ and all $a,b\in B$, there is $a'\in B$ such that $ay=ya'$, thus
 \[
 (xa)(yb)=xya'b\text{ belongs to }\Delta\,.
 \]
Therefore, $\Delta$ is an inverse subsemigroup of~$T$, and therefore so is the closure~$\Delta^{\oplus}$ of~$\Delta$ under finite orthogonal joins.

Let~$\bx$ be written as in~\eqref{Eq:CanFormBias}.
Then $\dd(\bx)=\bigoplus_{i=1}^n\dd(x_ia_i)=\bigoplus_{i=1}^n\dd(x_i)a_i$ belongs to~$B$.
It follows that $\Idp{\Delta^{\oplus}}\subseteq B$.
On the other hand, for every $a\in B$, there are $n\in\ZZ^+$ and $a_0,\dots,a_{n-1}\in S$ such that $a\leq\bigvee_{i<n}a_i$.
Setting $b_i=a\wedge(a_i\sd\bigvee_{j<i}a_j)$ for all $i<n$, we get
 \[
 a=\bigoplus_{i<n}b_i=\bigoplus_{i<n}a_ib_i\in
 \Delta^{\oplus}\,,
 \]
thus completing the proof that $B\subseteq\Idp{\Delta^{\oplus}}$.
Therefore, $B=\Idp{\Delta^{\oplus}}$.

Since~$\Delta^{\oplus}$ is an inverse subsemigroup of~$T$, closed under finite orthogonal joins and difference of idempotent elements, it is, by \cite[Corollary 3-2.7]{WBIS}, a sub-bias of~$T$.
Since it contains~$S$, it follows that $T=\Delta^{\oplus}$.
Hence, $\Idp{T}=\Idp{\Delta^{\oplus}}=B$.
\end{proof}

\begin{corollary}\label{C:CanFormBias}
Let~$S$ be a finite inverse semigroup.
Then the universal bias~$\Ub{S}$ is finite.
\end{corollary}

The following is a slight strengthening of Lemma~\ref{L:CanFormBias}.

\begin{lemma}\label{L:CanFormBias2}
Let~$S$ be an inverse subsemigroup of a Boolean inverse semigroup~$T$.
Suppose that~$T$ is generated by~$S$ as a bias.
Then every element $\bx\in T$ can be written in the form
\begin{equation}\label{Eq:CanFormBias2}
 \begin{aligned}
 &\bx=
 \bigoplus_{i=1}^nx_i\pII{a_i\sd\bigvee_{j=1}^{n_i}b_{i,j}}\,,\\
 &\text{where all }x_i\in S\,,\text{ all }a_i,b_{i,j}\in\Idp{S}\,,
 \text{ all }b_{i,j}\leq a_i\,,\text{ and all }a_i\leq\dd(x_i)\,.
 \end{aligned}
 \end{equation}
\end{lemma}
%

\begin{proof}
By virtue of Lemma~\ref{L:CanFormBias}, the Boolean ring $B=\Idp{T}$ is generated by~$\Idp{S}$.
Furthermore, by that lemma, it suffices to prove the existence of the given decomposition in case $\bx=xa$, where $x\in S$ and $a\in B$.
Since~$\Idp{S}$ is closed under finite meets, $a$ is a finite orthogonal join of elements of the form $a_i\sd\bigvee_{j=1}^{n_i}b_{i,j}$.
Replacing each~$a_i$ by~$a_i\dd(x_i)$ and each~$b_{i,j}$ by~$a_i\dd(x_i)b_{i,j}$, we get the desired conclusion.
\end{proof}

\begin{lemma}\label{L:FBH}
Let~$S$ be an inverse semigroup, let~$n\in\ZZ^+$, and let $x,x_1,\dots,x_n\in S$ such that the~$x_i$ are pairwise compatible.
Then $x\leq\bigvee_{i=1}^nx_i$, within~$\Ub{S}$, if{f} $x\leq x_i$ for some~$i$.
\end{lemma}

\begin{proof}
We prove the nontrivial direction.
Suppose that $x\leq\bigvee_{i=1}^nx_i$ within~$\Ub{S}$.
Denote by $\rho\colon(S,\cdot)\to(\fI_{S},\circ)$ the inverse semigroup embedding given by the Vagner-Preston Theorem, which we will call the \emph{Vagner-Preston completion of~$S$}: for every $z\in S$, $\rho_z$ is the bijection from~$\dd(z)S$ onto~$\rr(z)S$ given by
 \[
 \rho_z(t)=zt\quad\text{whenever}\quad t\in\dd(z)S\,.
 \]
{}From our assumption it follows that the partial function $\rho_x$ is extended by the union of all the~$\rho_{x_i}$.
Since $\dd(x)$ belongs to the domain of~$\rho_x$, it also belongs to the domain of~$\rho_{x_i}$ for some~$i$, and then $x=\rho_x(\dd(x))=\rho_{x_i}(\dd(x))=x_i\dd(x)$, that is, $x\leq x_i$.
\end{proof}

\begin{lemma}\label{L:BasicIneqUbS}
Let~$S$ be an inverse semigroup, let $m,n\in\ZZ^+$, and let $x$, $y$, $a$, $a_1$, \dots, $a_m$, $b_1$, \dots, $b_n$ be elements of~$S$, with~$a$, $b$, $a_1$, \dots, $a_m$, $b_1$, \dots, $b_n$ all idempotent.
Then $x\pI{a\sd\bigvee_{i=1}^ma_i}\leq y\pI{b\sd\bigvee_{j=1}^nb_j}$ within~$\Ub{S}$ if{f} the following statements hold:
 \begin{gather}
 \text{there is }i\in[1,m]\text{ such that }
 a\dd(x)\leq a_i\text{ or }
 \pI{xa=y\dd(x)a\text{ and }a\dd(x)\leq b\dd(y)}\,,
 \label{Eq:DisjCond}\\
 \text{for each }j\in[1,n]\text{ there is }i\in[1,m]
 \text{ such that }
 a\dd(x)b_j\leq a_i\,.\label{Eq:AEcond} 
 \end{gather}
 
\end{lemma}

\begin{proof}
By replacing~$a$ by~$a\dd(x)$ and~$b$ by~$b\dd(y)$, we may assume without loss of generality that $a\leq\dd(x)$ and $b\leq\dd(y)$.

Suppose first that the conditions~\eqref{Eq:DisjCond} and~\eqref{Eq:AEcond} both hold.
Set $u=a\sd\bigvee_{i=1}^ma_i$ and $v=b\sd\bigvee_{j=1}^nb_j$.
We must prove that $xu\leq yv$ within~$\Ub{S}$.

If $a\leq a_i$ for some~$i$, then $u=0$, thus $xu=0$ and we are done.
Suppose now that $a\nleq a_i$ for all~$i$.
It follows from~\eqref{Eq:DisjCond} that $a\leq b$ and $xa=ya$.
An elementary application of Lemma~\ref{Eq:A-BleqC-D}, together with~\eqref{Eq:DisjCond} and~\eqref{Eq:AEcond}, then yields $u\leq v$.
Since $xa=ya$ and $u\leq a$, it thus follows that $xu=yu\leq yv$.

Suppose, conversely, that $xu\leq yv$ within~$\Ub{S}$.
Denote again by $\rho\colon(S,\cdot)\to(\fI_{S},\circ)$ the Vagner-Preston completion of~$S$.
By projecting the equation $xu\leq yv$ onto the symmetric inverse semigroup~$\fI_{S}$, \emph{via}~$\rho$, we obtain
 \begin{equation}\label{Eq:VPeqn2}
 \rho_x\circ\pII{\id_{aS}\sd\bigcup_{i=1}^m\id_{a_iS}}
 \subseteq
 \rho_y\circ\pII{\id_{bS}\sd\bigcup_{j=1}^n\id_{b_jS}}\,,
 \end{equation}
where the containment symbol, between partial functions, means extension and the union symbol, applied to partial functions, means the least common extension.
Since $a\leq\dd(x)$ and $b\leq\dd(y)$, the left hand side and the right hand side of~\eqref{Eq:VPeqn2} have respective domains
 \[
 U=aS\setminus\bigcup_{i=1}^ma_iS\text{ and }
 V=bS\setminus\bigcup_{j=1}^nb_jS\,,
 \]
and it follows from~\eqref{Eq:VPeqn2} that $U\subseteq V$.
By applying Lemma~\ref{Eq:A-BleqC-D}, within the powerset lattice of~$S$, the latter containment implies that
 \begin{align}
 aS&\subseteq\bigcup_{i=1}^ma_iS\cup bS\,,
 \label{Eq:aaib}\\
 aS\cap\bigcup_{j=1}^nb_jS&\subseteq
 \bigcup_{i=1}^ma_iS\,.\label{Eq:abjai} 
 \end{align}
Observing that~$a$ belongs to the left hand side of~\eqref{Eq:aaib}, we obtain that
 \begin{gather}
 a\leq a_i\text{ for some }i\in[1,m]\text{ or }a\leq b\,,
 \label{Eq:aaib2}\\
 \text{for each }j\text{ there exists }i\text{ such that }
 ab_j\leq a_i\,.\label{Eq:abjai2}
 \end{gather}
Furthermore, suppose that $a\nleq a_i$ for all~$i$.
Then~$a$ belongs to~$U$, thus, by~\eqref{Eq:VPeqn2}, $\rho_x(a)=\rho_y(a)$, that is, $xa=ya$. 
\end{proof}

\begin{proposition}\label{P:BIS2IS}
Let~$\vp(\vx_1,\dots,\vx_n)$ and~$\vq(\vx_1,\dots,\vx_n)$ be terms in the similarity type of biases.
Then there is a positive quantifier-free formula $\vr(\vx_1,\dots,\vx_n)$, in the similarity type of inverse semigroups, such that for every inverse semigroup~$S$ and all elements $x_1,\dots,x_n\in S$,
$\Ub{S}$ satisfies $\vp(x_1,\dots,x_n)=\vq(x_1,\dots,x_n)$ if{f}~$S$ satisfies $\vr(x_1,\dots,x_n)$.
\end{proposition}

\begin{proof}
Set $\vec{\vx}=(\vx_1,\dots,\vx_n)$.
Expressing the equation $\vp(\vec{\vx})=\vq(\vec{\vx})$ as the conjunction of the two inequalities
$\vp(\vec{\vx})\leq\vq(\vec{\vx})$ and $\vq(\vec{\vx})\leq\vp(\vec{\vx})$, we see that it suffices to establish the conclusion for the inequality $\vp(\vec{\vx})\leq\vq(\vec{\vx})$.
 
Set $\gS=\set{\vx_1,\dots,\vx_n}$.
Due to the Vagner-Preston Theorem, every inverse semigroup embeds into a bias, thus the canonical map from~$\Fi{\gS}$ into~$\Fb{\gS}$ is one-to-one.
Applying Lemma~\ref{L:CanFormBias2} to the inclusion $\Fi{\gS}\hookrightarrow\Fb{\gS}$, we obtain, for every bias term $\vu(\vec{\vx})$, nonnegative integers~$m_{\vu}$ and~$n_{\vu,i}$ together with elements $\vs_i^{\vu}(\vec{\vx})$ in~$\Fi{\gS}$, for $1\leq i\leq m_{\vu}$, and idempotent elements $\va_i^{\vu}(\vec{\vx})$ and~$\vb_{i,j}^{\vu}(\vec{\vx})$ in~$\Fi{\gS}$, for $1\leq i\leq m_{\vu}$ and $1\leq j\leq n_{\vu,i}$, such that the relations
 \begin{align}
 \vb_{i,j}^{\vu}(\vec{\vx})&\leq\va_i^{\vu}(\vec{\vx})\,,
 \label{Eq:bijleqai}\\
 \va_i^{\vu}(\vec{\vx})&\leq\dd(\vs_i^{\vu}(\vec{\vx}))\,,
 \label{Eq:aiuleqsi}\\
 \vu(\vec{\vx})&
 =\bigoplus_{i=1}^{m_{\vu}}\vs_i^{\vu}(\vec{\vx})
 \pII{\va_i^{\vu}(\vec{\vx})\sd
 \bigvee_{j=1}^{n_{\vu,i}}\vb_{i,j}^{\vu}(\vec{\vx})}
 \label{Eq:Decompvu}
 \end{align}
all hold in~$\Fb{\gS}$, thus in every bias.

In particular, for every bias~$S$ and every finite sequence $\vec{x}=(x_1,\dots,x_n)\in S^n$, the inequality $\vp(\vec{x})\leq\vq(\vec{x})$ is equivalent to the conjunction of all inequalities
 \begin{equation}\label{Eq:SipleqJJsiq}
 \vs_i^{\vp}(\vec{x})
 \pII{\va_i^{\vp}(\vec{x})\sd
 \bigvee_{j=1}^{n_{\vp,i}}\vb_{i,j}^{\vp}(\vec{x})}\leq
 \bigvee_{k=1}^{m_{\vq}}\vs_{k}^{\vq}(\vec{x})
 \pII{\va_{k}^{\vq}(\vec{x})\sd
 \bigvee_{l=1}^{n_{\vq,k}}\vb_{k,l}^{\vq}(\vec{x})}\,,
 \end{equation}
for $1\leq i\leq m_{\vp}$.
By virtue of~\eqref{Eq:bijleqai} and~\eqref{Eq:aiuleqsi}, each inequality~\eqref{Eq:SipleqJJsiq} is, in turn, equivalent to the conjunction of all the inequalities
 \begin{gather}
 \va_i^{\vp}(\vec{x})\sd
 \bigvee_{j=1}^{n_{\vp,i}}\vb_{i,j}^{\vp}(\vec{x})\leq
 \bigvee_{k=1}^{m_{\vq}}\pII{\va_{k}^{\vq}(\vec{x})\sd
 \bigvee_{l=1}^{n_{\vq,k}}\vb_{k,l}^{\vq}(\vec{x})}\,,
 \label{Eq:aibijleqakbkj}\\
 \vs_i^{\vp}(\vec{x})
 \pII{\va_i^{\vp}(\vec{x})\sd
 \bigvee_{j=1}^{n_{\vp,i}}\vb_{i,j}^{\vp}(\vec{x})}
 \pII{\va_{i'}^{\vq}(\vec{x})\sd
 \bigvee_{j=1}^{n_{\vq,i'}}\vb_{i',j}^{\vq}(\vec{x})}\leq
 \vs_{i'}^{\vq}(\vec{x})\,,
 \label{Eq:aii'jj'leq}
 \end{gather}
where $1\leq i\leq m_{\vp}$ and $1\leq i'\leq m_{\vq}$.
By \cite[Lemma~5-2.13]{WBIS}, the inequality~\eqref{Eq:aibijleqakbkj} can be expressed by a conjunction of formulas of the form $\bigwedge_{k\in X}c_k\leq\bigvee_{k\notin X}c_k$, where the~$c_k$ are parameters, in~$\Idp S$, among the~$\va_i^{\vp}(\vec{x})$, $\vb_{i,j}^{\vp}(\vec{x})$, $\va_{i'}^{\vq}(\vec{x})$, $\vb_{i',j'}^{\vq}(\vec{x})$.
Since~$\Idp S$ is closed under finite meets (i.e., products), this reduces to a conjunction of formulas of the form $c\leq\bigvee_kc_k$, where~$c$ and the~$c_k$ are idempotent elements in~$\Idp S$.
By Lemma~\ref{L:FBH}, each such formula is equivalent to the disjunction of the formulas $c\leq c_k$.
Hence, we can express the inequality~\eqref{Eq:aibijleqakbkj} by a positive quantifier-free formula with parameters from~$\Idp{\Fi{\gS}}$.
By using Lemmas~\ref{Eq:A-BleqC-D} and~\ref{L:BasicIneqUbS}, so can the inequality~\eqref{Eq:aii'jj'leq}.
\end{proof}

We are now reaching the main result of this section.

\begin{theorem}\label{T:BISdecid}
Every free bias is residually finite.
In particular, the variety of all biases is generated by all finite symmetric biases~$\fI_N$, and the word problem for free biases is decidable.
\end{theorem}

\begin{proof}
Let~$\Sigma$ be an alphabet and let~$x$ and~$y$ be elements of the free bias~$\Fb{\Sigma}$ on~$\Sigma$, such that $x\neq y$.
We need to find a positive integer~$N$ and a bias homomorphism $\gf\colon\Fb{\Sigma}\to\fI_N$ such that $\gf(x)\neq\gf(y)$.
Write $x=\vp(x_1,\dots,x_n)$ and $y=\vq(x_1,\dots,x_n)$, for bias terms~$\vp$ and~$\vq$ and elements $x_1,\dots,x_n\in\Sigma$, and denote by~$\vr$ the positive quantifier-free formula associated to $(\vp,\vq)$ \emph{via} Proposition~\ref{P:BIS2IS}.
Now it follows from Proposition~\ref{P:BIS2IS} that the formula
 \[
 \vr(x_1,\dots,x_n)
 \]
does not hold in the free inverse semigroup~$\Fi{\Sigma}$.
Since, by Munn's Theorem~\cite{Munn1974}, $\Fi{\Sigma}$ is residually finite, there are a finite inverse semigroup~$T$ and a homomorphism $\psi\colon\Fi{\Sigma}\to T$ such that the formula
 \[
 \vr(\psi(x_1),\dots,\psi(x_n))
 \]
does not hold in~$T$.
Again by Proposition~\ref{P:BIS2IS} and denoting by $\ol{\psi}\colon\Fb{\Sigma}\to\Ub{T}$ the unique extension of~$\psi$ to a bias homomorphism, this means that
 \[
 \vr(\ol{\psi}(x_1),\dots,\ol{\psi}(x_n))
 \]
holds within~$\Ub{T}$;
that is, $\ol{\psi}(x)\neq\ol{\psi}(y)$ within~$\Ub{T}$.
Now by Corollary~\ref{C:CanFormBias}, $\Ub{T}$ is a finite bias, thus, by the Lawson-Lenz duality from~\cite{LaLe13} (see also \cite[Section 3-3]{WBIS}), there are a positive integer~$N$ and a bias embedding $\theta\colon\Ub{T}\hookrightarrow\nobreak\fI_N$.
Set $\gf=\theta\circ\ol{\psi}$.
Then $\gf(x)\neq\gf(y)$, which completes the proof that~$\Fb{\gS}$ is residually finite.
By McKinsey's work~\cite{McKin1943}, it follows that if~$\gS$ is finite, then the word problem for~$\Fb{\gS}$ is decidable.
\end{proof}

\section{Homogeneous sequences and rook matrices}
\label{S:Rook}

In this section we state an analogue, for biases, of the block matrix decomposition of an endomorphism of a module (Lemma~\ref{L:Mne1Se1}).
We also characterize, in Lemma~\ref{L:Typ=Z}, the Boolean inverse monoids with type monoid~$\ZZ^+$, and we describe in Proposition~\ref{P:BISFinIdp} the Boolean inverse monoids with finite sets of idempotents.

The terminology of the following definition is inspired by von~Neumann's work on regular rings.

\begin{definition}\label{S:Homog}
Let~$S$ be a Boolean inverse semigroup.
A finite sequence $(e_1,\dots,e_n)$ of idempotent elements of~$S$ is \emph{homogeneous} if $e_ie_j=0$ and $e_i\scD e_j$, whenever $i,j\in[1,n]$ with $i\neq j$.
\end{definition}

\begin{lemma}\label{L:Mne1Se1}
Let~$n$ be a positive integer, let~$S$ and~$T$ be Boolean inverse semigroups with~$S$ an additive ideal of~$T$, and let $(e_1,\dots,e_n)$ be a homogeneous sequence in~$T$.
Set $e=\bigoplus_{i=1}^ne_i$.
Then $eSe\cong\Matp{n}{e_1Se_1}$.
\end{lemma}

\begin{proof}
For each $i\in[1,n]$, we pick $c_i\in T$ such that $\dd(c_i)=e_1$ and $\rr(c_i)=\nobreak e_i$.
We may assume that $c_1=e_1$.
Let $\gf\colon eSe\to\Matp{n}{e_1Se_1}$ and $\psi\colon\Matp{n}{e_1Se_1}\to eSe$ be the maps given by
 \[
 \gf(x)=\begin{pmatrix}
 c_1^{-1}xc_1 & c_1^{-1}xc_2 & \dots & c_1^{-1}xc_n\\
 c_2^{-1}xc_1 & c_2^{-1}xc_2 & \dots & c_2^{-1}xc_n\\
 \vdots & \vdots & \ddots & \vdots\\
 c_n^{-1}xc_1 & c_n^{-1}xc_2 & \dots & c_n^{-1}xc_n
 \end{pmatrix}\,,\quad
 \text{for each }x\in eSe\,,
 \]
and
 \[
 \psi\pI{(x_{i,j})_{(i,j)\in[n]\times[n]}}=
 \bigoplus_{(i,j)\in[n]\times[n]}
 c_ix_{i,j}c_j^{-1}\,,\ 
 \text{for each }
 (x_{i,j})_{(i,j)\in[n]\times[n]}\in\Matp{n}{e_1Se_1}\,.
 \]
The verification that the maps~$\gf$ and~$\psi$ are well defined and mutually inverse semigroup isomorphisms is straightforward, if not a bit tedious.
\end{proof}

The \emph{Boolean unitization} of a Boolean inverse semigroup~$S$, introduced in \cite[Chapter~6]{WBIS}, is the unique (up to isomorphism) Boolean inverse monoid~$\widetilde{S}$ in which every element has the form $(1\sd e)\oplus x$, where $e\in\Idp S$ and $x\in S$, and such that $\widetilde{S}=S$ if~$S$ is unital.
In particular (cf. \cite[Proposition~6-6.5]{WBIS}), $S$ is an additive ideal of~$\widetilde{S}$ and if~$S$ is not unital, then $\widetilde{S}/S$ is the two-element inverse semigroup.

\begin{lemma}\label{L:MmnS}
Let~$S$ be a Boolean inverse semigroup and let~$m$ and~$n$ be positive integers.
Then $\Matp{mn}{S}\cong\Matp{n}{\Matp{m}{S}}$.
\end{lemma}

\begin{proof}
Observe, first, that $\Matp{mn}{S}$ is an additive ideal of $\Matp{mn}{\widetilde{S}}$.
Whenever\linebreak $1\leq i\leq mn$, we denote by~$a_i$ the diagonal matrix, in $\Matp{mn}{\widetilde{S}}$, with unique nonzero entry at~$(i,i)$, equal to~$1$.
Setting $e_k=\bigoplus_{i=1+(k-1)m}^{km}a_i$, the finite sequence $(e_1,\dots,e_n)$ is homogeneous in $\Matp{mn}{\widetilde{S}}$, and $e_1\Matp{mn}{S}e_1$ consists of all generalized rook matrices, over~$S$, all whose entries outside $[m]\times[m]$ are zero; thus it is isomorphic to
$\Matp{m}{S}$.
Apply Lemma~\ref{L:Mne1Se1}.
\end{proof}

\begin{lemma}\label{L:Typ=Z}
Let~$S$ be a Boolean inverse monoid and let~$m$ be a positive integer.
Then $(\Typ S,\typ_S(1))\cong(\ZZ^+,m)$ if{f} $S\cong\Matp{m}{\pz{G}}$ for some group~$G$.
\end{lemma}

\begin{proof}
We start with the case where $m=1$.
Since~$\pz{G}$ has exactly one nontrivial idempotent, we get the isomorphism $(\Typ\pz{G},\typ_{\pz{G}}(1))\cong(\ZZ^+,1)$.
Suppose, conversely, that $(\Typ S,\typ_S(1))\cong(\ZZ^+,1)$.
For any $a\in\Idp{S}$, $1=\typ_S(1)=\typ_S(a)+\typ_S(1-a)$, thus either $\typ_S(a)=0$ or $\typ_S(1-a)=0$, that is, either $a=0$ or $a=1$.
It follows that $G\eqdef S\setminus\set{0}$ is a group and $S\cong\pz{G}$.

Now we deal with the general case.
Since the idempotent elements of~$\Matp{m}{\pz{G}}$ form a Boolean algebra with~$m$ atoms, all pairwise $\scD$-equivalent (cf. Proposition~\ref{P:IdpMnS}), we get the isomorphism $\pI{\Typ\Matp{m}{\pz{G}},\typ(1)}\cong(\ZZ^+,m)$.
Let, conversely, $S$ be a Boolean inverse monoid such that $(\Typ S,\typ(1))\cong(\ZZ^+,m)$.
It follows from \cite[Lemma~4-1.6]{WBIS} that there is a decomposition $1=e_1\oplus\cdots\oplus e_m$ with each $\typ(e_i)=1$.
In particular, $(e_1,\dots,e_m)$ is a homogeneous sequence in~$S$.
By Lemma~\ref{L:Mne1Se1}, it follows that $S\cong\Matp{m}{e_1Se_1}$.
By (the proof of) \cite[Theorem~4-2.6]{WBIS}, the following isomorphism holds:
 \begin{equation}\label{Eq:TypeSe}
 \pI{\Typ(e_1Se_1),\typ_{e_1Se_1}(e_1)}\cong
 \pI{\ZZ^+,1}\,.
 \end{equation}
By the first part of the present proof, it follows that $e_1Se_1\cong\pz{G}$ for some group~$G$.
Therefore, $S\cong\Matp{m}{e_1Se_1}\cong\Matp{m}{\pz{G}}$.
\end{proof}

In \cite[Theorem~4.18]{Laws12}, Mark Lawson describes finite Boolean inverse monoids in terms of \emph{groupoids}.
The methods of the present section yield the following description of those monoids, and, more generally, of the Boolean inverse monoids with finite sets of idempotents, in terms of \emph{groups}.
Although we will not need this result in the rest of the paper, we found it worth recording here.

\begin{definition}\label{D:GroupMat}
A Boolean inverse monoid is \emph{fully group-matricial} if it is isomorphic to $\Matp{n}{\pz{G}}$, for some positive integer~$n$ and some group~$G$.
\end{definition}

\begin{proposition}\label{P:BISFinIdp}
Let~$S$ be a Boolean inverse monoid.
Then~$S$ has finitely many idempotent elements if{f} it is isomorphic to a finite product of fully group-matricial Boolean inverse semigroups.
\end{proposition}

\begin{proof}
We prove the nontrivial direction.
Assume that the Boolean algebra~$B$ of all idempotent elements of~$S$ is finite, denote by~$A$ the set of all its atoms, and denote by~$\bgq$ the restriction of Green's equivalence relation~$\scD$ to~$A$.
The elements $e_{\ba}=\bigvee\ba$, for $\ba\in A/{\bgq}$, are all idempotent, and they satisfy the relation
 \begin{equation}\label{Eq:ebapart}
 1=\bigoplus\vecm{e_{\ba}}{\ba\in A/{\bgq}}
 \text{ within }S\,.
 \end{equation}
We claim that~$e_{\ba}$ belongs to the center of~$S$, for every $\ba\in P/{\bgq}$.
Indeed, let $x\in S$.
Then for every $p\in\ba$, the element $xpx^{-1}$ is either zero or~$\bgq$-equivalent to~$p$, thus it belongs to $\ba\cup\set{0}$, and thus it lies below~$e_{\ba}$.
It follows that the element $xe_{\ba}x^{-1}=\bigoplus\vecm{xpx^{-1}}{p\in\ba}$ lies below~$e_{\ba}$.
Hence,
 \[
 xe_{\ba}=xx^{-1}xe_{\ba}=xe_{\ba}x^{-1}x\leq e_{\ba}x\,.
 \]
The proof that $e_{\ba}x\leq xe_{\ba}$ is symmetric.
This completes the proof of our claim.

By virtue of~\eqref{Eq:ebapart}, this yields a direct decomposition $S\cong\prod_{\ba\in P/{\bgq}}{e_{\ba}S}$, thus reducing the problem to the case where~$\bgq$ has exactly one equivalence class.
Thus, denoting by~$e_1$, \dots, $e_n$ the distinct atoms of~$B$, the finite sequence $(e_1,\dots,e_n)$ is homogeneous.
By Lemma~\ref{L:Mne1Se1}, $S\cong\Matp{n}{e_1Se_1}$.
Since~$e_1$ is the only nontrivial idempotent in~$e_1Se_1$, we get $e_1Se_1\cong\pz{G}$ for some group~$G$.
\end{proof}

Alternatively, Proposition~\ref{P:BISFinIdp} can also be obtained from \cite[Theorem~4.18]{Laws12}(1), by decomposing the groupoid obtained there into its connected components, then using Lemma~\ref{L:Mne1Se1} as above.

\section{Type monoids of finitely subdirectly irreducible biases}
\label{S:VarBIS}

The main aim of this section is to establish Lemma~\ref{L:1/2SIbias}, which states that the type monoid of every (finitely) subdirectly irreducible bias is \emph{prime} (cf. Definition~\ref{D:PrimeMon}).

Denote by~$\bgq_I$ the congruence generated by~$I\times\set{0}$, for an additive ideal~$I$ of a Boolean inverse semigroup~$S$ (cf. \cite[Proposition~3-4.6]{WBIS}; $\bgq_I$ is denoted there by~$\equiv_I$).
Recall that this congruence can be defined by
  \begin{equation}\label{Eq:defequivI}
  (x,y)\in\bgq_I\Leftrightarrow(\exists z)
 \pI{z\leq x\text{ and }z\leq y\text{ and }
 \set{x\sd z,y\sd z}\subseteq I}\,,
 \quad\text{for all }x,y\in S\,.
 \end{equation}

\begin{lemma}\label{L:thetaIcapJ}
$\bgq_I\cap\bgq_J=\bgq_{I\cap J}$, for all additive ideals~$I$ and~$J$ of~$S$.
\end{lemma}

\begin{proof}
It is sufficient to prove that $\bgq_I\cap\bgq_J\subseteq\bgq_{I\cap J}$.
Let $(x,y)\in\bgq_I\cap\bgq_J$.
By definition (cf.~\eqref{Eq:defequivI}), there are $u,v\leq x,y$ such that both containments $\set{x\sd u,y\sd u}\subseteq I$ and $\set{x\sd v,y\sd v}\subseteq J$ hold.
Since $u,v\leq x$, the elements~$u$ and~$v$ are compatible, thus they have a join~$w$, and $w\leq x$.
{}From $u,v\leq y$ it follows that $w\leq y$.
{}From $x\sd u=(x\sd w)\oplus(w\sd u)$ it follows that $x\sd w\leq x\sd u$, thus, since $x\sd u\in I$, we get $x\sd w\in I$.
Likewise, $x\sd w\in J$, so $x\sd w\in I\cap J$.
Likewise, $y\sd w\in I\cap J$, so~$w$ witnesses that $(x,y)\in\bgq_{I\cap J}$.
\end{proof}

If~$S$ is a \emph{Boolean inverse \msem}, that is, $x\wedge y$ exists for all $x,y\in\nobreak S$, then the satisfaction of~\eqref{Eq:defequivI} needs to be checked only on the element $z=x\wedge y$, which implies immediately that Lemma~\ref{L:thetaIcapJ} can be extended to arbitrary infinite collections of additive ideals.
However, the following example shows that this observation does not extend to the case where~$S$ need not be an inverse \msem.
Recall that an inverse semigroup is a \emph{Clifford inverse semigroup} if it satisfies the identity $\dd(\vx)=\rr(\vx)$ (i.e., $\vx^{-1}\vx=\vx\vx^{-1}$).

\begin{examplepf}\label{Ex:InfCollIdcap0}
A Clifford Boolean inverse monoid~$S$, with an infinite descending sequence $\vecm{I_n}{n\in\ZZ^+}$ of additive ideals such that $\bigcap_{n\in\ZZ^+}I_n=\set{0}$, yet $\bigcap_{n\in\ZZ^+}\bgq_{I_n}$ is not the identity congruence.
\end{examplepf}

\begin{proof}
The example in question is the one of \cite[Example~3-3.5]{WBIS}.
Let us recall its construction.
Denote by~$\cB$ the Boolean algebra of all subsets of~$\NN$ that are either finite or cofinite, and pick any nontrivial group~$G$.
For every $x\in\cB$, we set $N_x=G$ if~$x$ is finite, and $N_x=\set{1}$ if~$x$ is cofinite.
For $g,h\in G$ and $x\in\cB$, let $g\equiv_xh$ hold if $g\equiv h\pmod{N_x}$.
We define an equivalence relation~$\sim$ on~$\cB\times G$ by setting
 \[
 (x,g)\sim(y,h)\quad\text{if}\quad
 (x=y\text{ and }g\equiv_xh)\,,
 \quad\text{for all }x,y\in\cB\text{ and all }g,h\in G\,,
 \]
and we denote by $[x,g]$ the $\sim$-equivalence class of $(x,g)$.
Then~$\sim$ is a semigroup congruence on~$\cB\times G$, and the quotient $S=(\cB\times G)/{\sim}$ is a Boolean inverse monoid, where $\dd[x,g]=\rr[x,g]=[x,1]$ whenever $(x,g)\in\cB\times G$.

For the rest of the proof, we pick any element~$g\in G\setminus\set{1}$, and we set
 \begin{align*}
 a_n&=[\NN\setminus[n],g]\,,\\
 e_n&=[\NN\setminus[n],1]\,,
 \end{align*}
for every $n\in\ZZ^+$.
The set $I_n\eqdef\setm{[x,h]\in S}{x\cap[n]=\es}$ is an additive ideal of~$S$, for every $n\in\ZZ^+$.
Obviously, $\bigcap_{n\in\ZZ^+}I_n=\set{0}$.
On the other hand, for every $n\in\ZZ^+$, both elements $a_0\sd[[n],1]=a_n$ and $e_0\sd[[n],1]=e_n$ belong to~$I_n$, thus the pair $(a_0,b_0)$ belongs to the intersection of all~$\bgq_{I_n}$, with $a_0\neq b_0$.
\end{proof}

\begin{definition}\label{D:PrimeMon}
A conical commutative monoid~$M$ is \emph{prime} if $M\setminus\set{0}$ is downward directed.
\end{definition}

\begin{lemma}[folklore]\label{L:OrthRefMon}
Let two elements~$x$ and~$y$ in a conical refinement monoid~$M$ be \emph{orthogonal}, in notation $x\perp y$, if there is no nonzero $z\in M$ such that $z\lep x$ and $z\lep y$.
Then the following statements hold:
\begin{enumeratei}
\item $x\perp z$ and $y\perp z$ implies that $x+y\perp z$, for all $x,y,z\in M$.

\item $x\perp y$ implies that $nx\perp ny$, for all $x,y\in M$ and every positive integer~$n$.
\end{enumeratei}
\end{lemma}

A bias~$S$ is \emph{subdirectly irreducible} (resp., \emph{finitely subdirectly irreducible}) if it has a smallest nonzero congruence (resp., if any two nonzero congruences of~$S$ have nonzero intersection).
Trivially, every subdirectly irreducible bias is finitely subdirectly irreducible.
By using the results of~\cite{WBIS}, it is easy to construct finitely subdirectly irreducible Boolean inverse monoids that are not subdirectly irreducible.

\begin{lemma}\label{L:1/2SIbias}
Let~$S$ be a finitely subdirectly irreducible bias.
Then the type monoid $\Typ S$ is prime.
\end{lemma}

\begin{proof}
It suffices to prove that for any nonzero idempotent elements~$a$ and~$b$ of~$S$, there exists a nonzero element of~$\Typ S$ below~$\typ_S(a)$ and~$\typ_S(b)$.
Denote by~$I(x)$ the additive ideal of~$S$ generated by~$\set{x}$, for any $x\in\nobreak S$.
Then~$\bgq_{I(a)}$ and~$\bgq_{I(b)}$ are both nonzero congruences of~$S$, thus, since~$S$ is finitely subdirectly irreducible, the intersection $\bgq_{I(a)}\cap\bgq_{I(b)}$ is a nonzero congruence of~$S$.
By Lemma~\ref{L:thetaIcapJ}, we get $I(a)\cap I(b)\neq\set{0}$.
The subsets
 \begin{align*}
 \bI_a&=\setm{\bx\in\Typ S}{(\exists n\in\NN)
 (\bx\lep n\cdot\typ_S(a))}\,,\\
 \bI_b&=\setm{\bx\in\Typ S}{(\exists n\in\NN)
 (\bx\lep n\cdot\typ_S(b))}
 \end{align*}
of~$\Typ S$ are both o-ideals of~$\Typ S$.
By \cite[Proposition~4-2.4]{WBIS}, the subsets
 \begin{align*}
 J(a)&=\setm{x\in S}{\typ_S(x)\in\bI_a}\,,\\
 J(b)&=\setm{x\in S}{\typ_S(x)\in\bI_b}
 \end{align*}
are both additive ideals of~$S$.
Since $a\in J(a)$ and $b\in J(b)$, it follows that $I(a)\subseteq J(a)$ and $I(b)\subseteq J(b)$.
(\emph{Actually, with a small additional effort, it is not hard to get $I(a)=J(a)$ and $I(b)=J(b)$.})
Hence, $J(a)\cap J(b)\neq\set{0}$, and hence there are a positive integer~$n$ and $\bc\in(\Typ S)\setminus\set{0}$ such that $\bc\lep n\typ_S(a),n\typ_S(b)$.
Since~$\Typ S$ is a conical refinement monoid, it has, by Lemma~\ref{L:OrthRefMon}, a nonzero element below $\typ_S(a)$ and $\typ_S(b)$.
\end{proof}

\section{Generators for varieties of biases}
\label{S:GenVarBias}

The main aim of this section is to prove that every proper variety of biases is generated by biases of generalized rook matrices, of finite order, over groups with zero (Theorem~\ref{T:VarMnG}).

Recall (see, for example, \cite{WBIS}) that the \emph{index} of an element~$e$ in a conical commutative monoid~$M$ is the least nonnegative integer~$n$, if it exists, such that $(n+1)x\lep e$ implies that $x=0$ whenever $x\in M$, and~$\infty$ otherwise.

\begin{lemma}\label{L:FinInd2Z}
Let~$M$ be a nonzero conical refinement monoid.
If~$M$ is prime and every element of~$M$ has finite index, then $M\cong\ZZ^+$.
\end{lemma}

\begin{proof}
We claim that every element of $M\setminus\set{0}$ is above an atom of~$M$.
Suppose otherwise and let $a\in M$ without any atom below~$a$.
Let $x\in(0,a]$.
Since there is no atom below~$x$, $x$ is not an atom, thus, since $x\neq0$, there are $x_0,x_1\in M\setminus\set{0}$ such that $x=x_0+x_1$.
Since~$M$ is prime, there is $y\neq0$ such that $y\lep x_0$ and $y\lep x_1$.
It follows that $2y\lep x$.
Arguing inductively, we find $n\in\ZZ^+$ and $z\in(0,a]$ such that~$2^n$ is greater than the index of~$a$ and $2^nz\lep a$, in contradiction with the definition of the index.
This proves our claim.

In particular, there is at least one atom~$p$ in~$M$.
Since~$M$ is prime, $p$ is the only atom in~$M$.
By the claim above, every element of~$M\setminus\set{0}$ is larger than or equal to~$p$.

Now let $a\in M$.
Since $p\lep x$ whenever $x\in(0,a]$, the index~$m$ of~$a$ is the largest integer such that $mp\lep a$.
Let~$b$ such that $mp+b=a$.
If $b\neq0$, then $p\lep b$, thus $(m+1)p\lep mp+b=a$, in contradiction with the definition of the index.
Therefore, $mp=a$, thus completing the proof that $M=\ZZ^+p$.
Since~$M$ is a conical refinement monoid and~$p$ is an atom, it follows that the map ($\ZZ^+\to M$, $n\mapsto np$) is one-to-one.
Therefore, $M\cong\ZZ^+$.
\end{proof}

\begin{notation}\label{Not:VarBias}
We denote by~$\Bis$ the variety of all biases.
\end{notation}

Recall that a partially ordered Abelian group~$G$ is \emph{Archimedean} if for all $a,b\in G$, if $na\leq b$ for every $n\in\ZZ^+$, then $a\leq0$.

\begin{theorem}\label{T:VarMnG}
Let~$\cV$ be a proper variety of biases.
\begin{enumerater}
\item\label{BVbdedInd}
There is a largest nonnegative integer~$h$ such that $\fI_h\in\cV$.
Furthermore, $h$ is the largest possible value of the index of~$\typ_S(e)$ within~$\Typ S$, for $S\in\cV$ and  $e\in\Idp{S}$.

\item\label{BVdimgrp}
For any $S\in\cV$, every element of the type monoid~$\Typ S$ has finite index.
In particular, $\Typ S$ is the positive cone of an Archimedean dimension group.

\item\label{GenVarBias}
The variety~$\cV$ is generated by the collection of all its fully group-matricial members.

\end{enumerater}
\end{theorem}

\begin{proof}
\emph{Ad}~\eqref{BVbdedInd}, \eqref{BVdimgrp}.
If~$\fI_n\in\cV$ for all $n\in\ZZ^+$, then, by
Theorem~\ref{T:BISdecid}, $\cV=\Bis$, which contradicts our assumption.
Hence there is a largest nonnegative integer~$h$ such that $\fI_h\in\cV$.
Moreover, either $h=0$ and $\Typ\fI_h=\set{0}$, or $h>0$ and $(\Typ\fI_h,\typ_{\fI_h}(1))\cong(\ZZ^+,h)$.
It follows that the index of~$\typ_{\fI_h}(1)$ in~$\Typ\fI_h$ is exactly~$h$.

Let~$S\in\cV$.
We shall prove that the index of $\be\eqdef\typ_S(e)$ is less than or equal to~$h$, for every $e\in\Idp{S}$.
Suppose otherwise.
By the definition of the index and by \cite[Lemma~4-1.6]{WBIS}, there are nonzero pairwise orthogonal idempotents $e_0,\dots,e_h\leq e$ such that $\typ_S(e_0)=\typ_S(e_i)$ whenever $0\leq i\leq h$.
By definition, $(e_0,\dots,e_h)$ is a homogeneous sequence of~$S$.
Set $e'=\bigoplus_{i=0}^he_i$.
By Lemma~\ref{L:Mne1Se1}, $\Matp{h+1}{e_0Se_0}$ is isomorphic to~$e'Se'$, hence it embeds into~$S$.
Since~$\fI_{h+1}$ embeds into~$\Matp{h+1}{e_0Se_0}$, it also belongs to~$\cV$, a contradiction.

Since every element of the type interval~$\Int S$ of~$S$ has finite index in~$\Typ S$, and since~$\Int S$ generates the refinement monoid~$\Typ S$, it follows from \cite[Corollary~3.12]{WDim} that every element of~$\Typ S$ has finite index.
By \cite[Proposition~3.13]{WDim}, it follows that~$\Typ S$ is the positive cone of an Archimedean dimension group.

\emph{Ad}~\eqref{GenVarBias}.
Denote by~$\cK$ the class of all fully group-matricial members of~$\cV$.
Every member~$S$ of~$\cV$ is the directed union, thus the direct limit, of all unital biases $eSe$, where $e\in\Idp S$; thus the unital members of~$\cV$ generate~$\cV$.
Further, every unital member~$S$ of~$\cV$ is a subdirect product of subdirectly irreducible members of~$\cV$, which are all homomorphic images of~$S$, thus they are all unital.
Hence, $\cV$ is generated by the class of its unital subdirectly irreducible members, so it suffices to prove that every unital subdirectly irreducible member~$S$ of~$\cV$ belongs to~$\cK$.

By Lemma~\ref{L:1/2SIbias}, $\Typ S$ is a prime conical refinement monoid.
Further, by~\eqref{BVdimgrp} above, every element of~$\Typ S$ has finite index in~$\Typ S$.
By Lemma~\ref{L:FinInd2Z}, it follows that $\Typ S\cong\ZZ^+$.
Identifying~$\Typ S$ with~$\ZZ^+$ and setting $n=\typ_S(1)$, it follows from Lemma~\ref{L:Typ=Z} that $S\cong\Matp{n}{\pz{G}}$ for some group~$G$.
\end{proof}

\section{Generalized rook matrices and wreath products of groups}
\label{S:GRMgp}

The main aim of this section is to relate embedding properties of fully group-matricial biases and embedding properties of the corresponding groups (Lemma~\ref{L:GBEmbed2}).
Owing to Lemma~\ref{L:InvMnG}, the latter will be stated in terms of wreath products by finite symmetric groups.

The proof of the following lemma is a straightforward application of \cite[Section~3.5]{WBIS} together with the equivalence between bias homomorphism and additive semigroup homomorphism, and we leave it to the reader.

\begin{lemma}\label{L:MnHom}
Let~$S$ and~$T$ be Boolean inverse semigroups and let~$n$ be a positive integer.
Then for every bias homomorphism $f\colon S\to T$, the assignment $\Matp{n}{f}\colon(x_{i,j})_{(i,j)\in[n]\times[n]}\mapsto\pI{f(x_{i,j})}_{(i,j)\in[n]\times[n]}$ defines a bias homomorphism\linebreak $\Matp{n}{f}\colon\Matp{n}{S}\to\Matp{n}{T}$.
Furthermore, $\Matp{n}{f}$ is one-to-one \pup{resp., surjective} if{f}~$f$ is one-to-one \pup{resp., surjective}.
\end{lemma}

\begin{lemma}\label{L:CongMnS}
Let~$S$ be a Boolean inverse semigroup and let~$n$ be a positive integer.
Then every additive congruence~$\bga$ of~$S$ gives rise to a unique additive congruence~$\Matp{n}{\bga}$ of~$\Matp{n}{S}$ such that
 \begin{equation}\label{Eq:DefnMnbga}
 (x_{i,j})_{(i,j)\in[n]\times[n]}\equiv_{\Matp{n}{\bga}}
 (y_{i,j})_{(i,j)\in[n]\times[n]}
 \ \Leftrightarrow\ \pI{x_{i,j}\equiv_{\bga}y_{i,j}
 \text{ for all }(i,j)\in[n]\times[n]}\,,
 \end{equation}
for all $(x_{i,j})_{(i,j)\in[n]\times[n]},(y_{i,j})_{(i,j)\in[n]\times[n]}\in\Matp{n}{S}$.
Furthermore, the canonical surjective homomorphism $\Matp{n}{\ga}\colon\Matp{n}{S}\twoheadrightarrow\Matp{n}{S/{\bga}}$ factors through a unique isomorphism $\Matp{n}{S}/{\Matp{n}{\bga}}\to\Matp{n}{S/{\bga}}$.
Conversely, every additive congruence of~$\Matp{n}{S}$ has the form~$\Matp{n}{\bga}$ for a unique additive congruence~$\bga$ of~$S$.
 
In particular, the assignment $\bga\mapsto\Matp{n}{\bga}$ defines an isomorphism from $\Con S$ onto $\Con\Matp{n}{S}$.
\end{lemma}

\begin{proof}
The canonical projection $\ga\colon S\twoheadrightarrow S/{\bga}$ is a bias homomorphism, which, by Lemma~\ref{L:MnHom}, induces a bias homomorphism $\Matp{n}{\ga}\colon\Matp{n}{S}\twoheadrightarrow\Matp{n}{S/{\bga}}$.
The kernel~$\Matp{n}{\bga}$ of that homomorphism is an additive congruence of~$\Matp{n}{S}$, and it is given by~\eqref{Eq:DefnMnbga}.
Observe that
 \[
 x_{(i',j')}=a_{(i',i)}x_{(i,j)}a_{(j,j')}\,,
 \]
for all $i,i',j,j'\in[n]$ and all $x,a\in S$ with~$a$ idempotent and $\dd(x)\vee\rr(x)\leq a$.

Now let~$\bgq$ be an additive congruence of~$\Matp{n}{S}$.
The equivalence relation~$\bga$ on~$S$ defined by
 \[
 x\equiv_{\bga}y\ \Leftrightarrow\ 
 x_{(1,1)}\equiv_{\bgq}y_{(1,1)}\,,
 \qquad\text{for all }x,y\in S\,,
 \]
is an additive congruence of~$S$, and it follows from the above that
 \begin{equation}\label{Eq:xijiff11}
 x\equiv_{\bga}y\ \Leftrightarrow\ 
 x_{(i,j)}\equiv_{\bgq}y_{(i,j)}\,,\qquad\text{for all }x,y\in S
 \text{ and all }i,j\in[n]\,.
 \end{equation}
We claim that $\bgq=\Matp{n}{\bga}$.
Let $x=(x_{i,j})_{(i,j)\in[n]\times[n]}$ and $y=(y_{i,j})_{(i,j)\in[n]\times[n]}$ in $\Matp{n}{S}$.

Suppose first that $x\equiv_{\Matp{n}{\bga}}y$.
For all $(i,j)\in[n]\times[n]$, the relation $x_{i,j}\equiv_{\bga}y_{i,j}$ holds, that is, by definition and by~\eqref{Eq:xijiff11}, $(x_{i,j})_{(i,j)}\equiv_{\bgq}(y_{i,j})_{(i,j)}$.
Since $x=\bigoplus_{(i,j)\in[n]\times[n]}(x_{i,j})_{(i,j)}$ and similarly for~$y$, it follows that $x\equiv_{\bgq}y$.

Suppose, conversely, that $x\equiv_{\bgq}y$ and let~$e$ be an idempotent element of~$S$ such that $\bigvee_{i,j}\dd(x_{i,j})\vee\bigvee_{i,j}\rr(x_{i,j})\leq e$.
Let $i,j\in[n]$.
{}From $x\equiv_{\bgq}y$ it follows that $e_{(i,i)}xe_{(j,j)}\equiv_{\bgq}e_{(i,i)}ye_{(j,j)}$, that is, $(x_{i,j})_{(i,j)}\equiv_{\bgq}(y_{i,j})_{(i,j)}$, thus, by~\eqref{Eq:xijiff11}, $x_{i,j}\equiv_{\bga}y_{i,j}$.
Therefore, $x\equiv_{\Matp{n}{\bga}}y$, thus completing the proof of our claim.

Finally, since the map $\bga\mapsto\Matp{n}{\bga}$ is clearly one-to-one, it defines an isomorphism from~$\Con S$ onto $\Con\Matp{n}{S}$.
\end{proof}

Taking $S=\pz{G}$ for a group~$G$, we get two types of congruences in~$S$:

\begin{itemize}
\item[(1)]
The congruences of the form $\pz{\bgq}=\bgq\cup\set{(0,0)}$, for a congruence~$\bgq$ of the group~$G$;

\item[(2)]
The full congruence $\one_G=\pz{G}\times\pz{G}$.
\end{itemize}

In turn, the congruences of the group~$G$ are in one-to-one correspondence with the normal subgroups of~$G$.
Denoting by~$\NSub G$ the lattice of all normal subgroups of~$G$, we thus obtain the following corollary to Lemma~\ref{L:CongMnS}.

\begin{corollary}\label{C:CongMnS}
Let~$G$ be a group and let~$n$ be a positive integer.
Then $\Con\pz{G}\cong\Con\Matp{n}{\pz{G}}\cong(\NSub G)\sqcup\set{\infty}$.
\end{corollary}

Our next series of lemmas will focus on bias homomorphisms between fully group-matricial biases.

\begin{lemma}\label{L:InvMnG}
Let~$G$ be a group and let~$n$ be a positive integer.
Then the group of all invertible elements of~$\Matp{n}{\pz{G}}$ is isomorphic to the wreath product $G\wr\fS_n$.
\end{lemma}

\begin{proof}
The group homomorphism $G\to\set{1}$ extends to an additive semigroup homomorphism, that is, a bias homomorphism, $\pz{G}\to\set{0,1}$, which, by Lemma~\ref{L:MnHom}, extends to a bias homomorphism~$\pi$ from $\Matp{n}{\pz{G}}$ to $\fI_n\cong\Matp{n}{\set{0,1}}$.
Let $x=(x_{i,j})_{(i,j)\in[n]\times[n]}$ in~$\Matp{n}{\pz{G}}$.
If~$x$ is invertible in~$\Matp{n}{\pz{G}}$, then~$\pi(x)$ is invertible in~$\fI_n$, thus it is a permutation matrix, that is, denoting by~$\fS_n$ the group of all permutations of~$[n]$,
 \[
 \pi(x)=(\delta_{i,\gs(j)})_{(i,j)\in[n]\times[n]}\,,\quad
 \text{for some }\gs\in\fS_n\,,
 \]
where~$\delta$ denotes Kronecker's symbol.
It follows that there is a finite sequence $(g_1,\dots,g_n)\in G^n$ such that
 \begin{equation}\label{Eq:deltag}
 x=(\delta_{i,\gs(j)}g_i)_{(i,j)\in[n]\times[n]}\,.
 \end{equation}
Denoting by~$[g_1,\dots,g_n;\gs]$ the right hand side of~\eqref{Eq:deltag}, the product of two such elements is given by
 \[
 [g_1,\dots,g_n;\ga]\cdot[h_1,\dots,h_n;\gb]=
 [g_1h_{\ga^{-1}(1)},\dots,g_nh_{\ga^{-1}(n)},\ga\gb]\,,
 \]
so the elements of the form~$[g_1,\dots,g_n;\gs]$ form a subgroup of the monoid~$\Matp{n}{\pz{G}}$, isomorphic to the wreath product $G\wr\fS_n$.
By the above, this subgroup contains all invertibles of~$\Matp{n}{\pz{G}}$, thus it consists exactly of all invertibles of~$\Matp{n}{\pz{G}}$.
\end{proof}

\begin{lemma}\label{L:GBEmbed}
Let~$G$ and~$H$ be groups and let~$n$ be a positive integer.
Then~$\pz{G}$ embeds, as a bias, into~$\Matp{n}{\pz{H}}$, if{f}~$G$ embeds, as a subgroup, into the wreath product $H\wr\fS_n$.
\end{lemma}

\begin{proof}
By Lemma~\ref{L:InvMnG}, any group embedding of~$G$ into $H\wr\fS_n$ gives rise to a group embedding of~$G$ into the group of all invertible elements of~$\Matp{n}{\pz{H}}$, which in turn extends to a bias embedding from~$\pz{G}$ into~$\Matp{n}{\pz{H}}$.

Let, conversely, $\gf\colon\pz{G}\hookrightarrow\Matp{n}{\pz{H}}$ be a bias embedding.
Since~$\gf(1)$ is an idempotent element of~$\Matp{n}{\pz{H}}$, it is, by Proposition~\ref{P:IdpMnS} a diagonal matrix with entries in $\set{0,1}$.
Denote by~$\gO$ the set of all indices $i\in[n]$ such that the $(i,i)$th entry of~$\gf(1)$ is~$1$.
For every $g\in G$, the element  $\gf(g)=\gf(1)\gf(g)\gf(1)$ belongs to $\gf(1)\Matp{n}{\pz{H}}\gf(1)$, which consists of all generalized rook matrices all whose entries outside $\gO\times\gO$ are zero.
Hence, setting $m=\card\gO$, the map~$\gf$ induces a unital bias embedding $\psi\colon\pz{G}\hookrightarrow\Matp{m}{\pz{H}}$.
Since~$\psi$ sends the unit to the unit, it sends every invertible element to an invertible element.
By Lemma~\ref{L:InvMnG}, $G$ embeds, as a subgroup, into $H\wr\fS_m$.
The latter embeds, in turn, into $H\wr\fS_n$, \emph{via}
$[h_1,\dots,h_m;\gs]\mapsto[h_1,\dots,h_m,1,\dots,1;\ol{\gs}]$, where~$\ol{\gs}$ stands for the extension of~$\gs$ by the identity map on $[n]\setminus[m]$.
\end{proof}

Denote by~$\ip{x}$ the largest integer less than or equal to~$x$, for any rational number~$x$.
The following lemma strengthens Lemma~\ref{L:GBEmbed} to bias embeddings between fully group-matricial biases.

\begin{lemma}\label{L:GBEmbed2}
Let~$G$ and~$H$ be groups, let~$m$ and~$n$ be positive integers.
Then~$\Matp{m}{\pz{G}}$ embeds, as a bias, into~$\Matp{n}{\pz{H}}$, if{f} $m\leq n$ and~$G$ embeds, as a subgroup, into the wreath product $H\wr\fS_{\ip{n/m}}$.
\end{lemma}

\begin{proof}
Suppose, first, that $m\leq n$ and that~$G$ embeds, as a subgroup, into\linebreak $H\wr\fS_{\ip{n/m}}$.
By Lemma~\ref{L:GBEmbed}, $\pz{G}$ embeds, as a bias, into $\Matp{\ip{n/m}}{\pz{H}}$.
It follows that $\Matp{m}{\pz{G}}$ embeds, as a bias, into $\Matp{m}{\Matp{\ip{n/m}}{\pz{H}}}$, thus, by Lemma~\ref{L:MmnS}, into $\Matp{m\ip{n/m}}{\pz{H}}$.
Since $m\ip{n/m}\leq n$, it follows that $\Matp{m}{\pz{G}}$ embeds, as a bias, into $\Matp{n}{\pz{H}}$.

Let, conversely, $\gf\colon\Matp{m}{\pz{G}}\hookrightarrow\Matp{n}{\pz{H}}$ be a bias embedding.
Whenever $1\leq i\leq m$, denote by~$a_i$ the generalized rook matrix of order~$m$ with $(i,i)$th entry~$1$ and all others~$0$.
Similarly, for $X\subseteq[n]$, denote by~$b_X$ the diagonal generalized rook matrix of order~$n$ with $(j,j)$th entry~$1$ if $j\in X$, $0$ otherwise.
Then each~$\gf(a_i)$ is a nonzero idempotent element of~$\Matp{n}{\pz{H}}$, so, by Proposition~\ref{P:IdpMnS}, $\gf(a_i)=b_{X_i}$ for some nonempty $X_i\subseteq[n]$.
Since $(a_1,\dots,a_m)$ is a homogeneous sequence in $\Matp{m}{\pz{G}}$, $(b_{X_1},\dots,b_{X_m})$ is a homogeneous sequence in $\Matp{n}{\pz{H}}$, thus the~$X_i$ are pairwise disjoint and they all have the same cardinality, say~$d$.
Hence, $md\leq n$.
Moreover, $\gf$ embeds $a_1\Matp{m}{\pz{G}}a_1$, which is isomorphic to~$\pz{G}$, into $b_{X_1}\Matp{n}{\pz{H}}b_{X_1}$, which is isomorphic to $\Matp{d}{\pz{H}}$.
By Lemma~\ref{L:GBEmbed}, $G$ embeds into $H\wr\fS_d$.
Since $d\leq\ip{n/m}$, the desired conclusion follows.
\end{proof}

\section{A projectivity property of fully group-matricial biases}
\label{S:Proj}

The main aim of this section is the projectivity property of fully group-matricial biases, within the class of all $\scD$-cancellative biases, stated in Lemma~\ref{L:CancProj}.
Our first lemma is an analogue, for biases, of the lattice-theoretical result stating the projectivity of von~Neumann frames (cf. \cite{Freese76,Huhn72}).

\begin{lemma}\label{L:InProj}
The symmetric inverse monoid~$\fI_n$ is a projective bias, for every positive integer~$n$.
\end{lemma}

\begin{proof}
Let~$S$ be a Boolean inverse semigroup.
We need to prove that every surjective bias homomorphism $\gf\colon S\twoheadrightarrow\fI_n$ has a right inverse bias embedding.
The kernel~$\bgf$ of~$\gf$ is an additive congruence of~$S$.
Denote by~$e_{i,j}$ the unique function from~$\set{j}$ to~$\set{i}$, whenever $i,j\in[n]$.
Since~$\gf$ is a surjective homomorphism of inverse semigroups, for every $i\in[n]$, there exists $a_{1,i}\in S$ such that $\gf(a_{1,i})=e_{1,i}$.
Since~$e_{1,1}$ is idempotent, we may take~$a_{1,1}$ idempotent.
By replacing each~$a_{1,i}$ by $a_{1,i}\pI{\dd(a_{1,i})\sd\bigvee_{1\leq j<i}\dd(a_{1,j})}$, we may assume that $\dd(a_{1,i})\dd(a_{1,j})=0$, that is,
 \begin{equation}\label{Eq:a1ia1j-1}
 a_{1,i}a_{1,j}^{-1}=0\,,\quad\text{for all distinct }i,j\in[n]\,.
 \end{equation}
Now set
 \begin{align*}
 b_{1,1}&\eqdef\bigwedge_{j=1}^n\rr(a_{1,j})\,,\\
 b_{i,i}&\eqdef a_{1,i}^{-1}b_{1,1}a_{1,i}\,,\\
 b_{1,i}&\eqdef b_{1,1}a_{1,i}\,, 
 \end{align*}
for each $i\in[n]$.
This causes no conflict of notation, because $b_{1,1}\leq a_{1,1}$.
All elements~$b_{i,i}$ are idempotent, and, by~\eqref{Eq:a1ia1j-1}, they are pairwise orthogonal.
Furthermore, for every~$i\in[n]$, it is easy to verify that $\gf(b_{i,i})=e_{i,i}$, $\gf(b_{1,i})=e_{1,i}$, $\dd(b_{1,i})=b_{i,i}$, and $\rr(b_{1,i})=b_{1,1}$.
Since $b_{1,i}\leq a_{1,i}$ and by~\eqref{Eq:a1ia1j-1}, we get 
 \begin{equation*}
 b_{1,i}b_{1,j}^{-1}=0\,,\quad\text{for all distinct }i,j\in[n]\,.
 \end{equation*}
We set $b_{i,j}\eqdef b_{1,i}^{-1}b_{1,j}$, whenever $i,j\in[n]$.
Then $\gf(b_{i,j})=e_{1,i}^{-1}e_{1,j}=e_{i,1}e_{1,j}=e_{i,j}$.
Furthermore, by using the above, it is not hard to verify that the~$b_{i,j}$ form a system of matrix units in~$S$, that is, $b_{i,j}b_{k,l}=\delta_{j,k}b_{i,l}$, whenever $i,j,k,l\in[n]$.
The map $\psi\colon\fI_n\to S$, $x\mapsto\bigoplus_{i\in\dom(x)}b_{x(i),i}$ is an additive semigroup homomorphism, and $\gf\circ\psi=\id_{\fI_n}$.
\end{proof}

In order to establish an analogue of Lemma~\ref{L:InProj} for fully group-matricial biases, we will need to add to our assumptions a statement of $\scD$-cancellativity.
We first prove a crucial preparatory lemma.

\begin{lemma}\label{L:LiftInv}
Let~$\bgq$ be an additive congruence of a $\scD$-cancellative Boolean inverse semigroup~$S$, let $a,b\in\Idp S$ such that $a\scD_S b$, and let~$\bx\in S/{\bgq}$ such that $\dd(\bx)=a/{\bgq}$ and $\rr(\bx)=b/{\bgq}$.
Then there exists $x\in\bx$ such that $\dd(x)=a$ and $\rr(x)=b$.
\end{lemma}

\begin{proof}
Pick $y\in\bx$.
It follows from our assumptions that $\dd(y)\equiv_{\bgq}a$ and $\rr(y)\equiv_{\bgq}b$.
Set $u\eqdef\dd(y)$ and $v\eqdef\rr(y)$.
The elements $u'\eqdef uay^{-1}by$ and $v'\eqdef yu'y^{-1}$ are idempotent, with $u'\leq ua$, $v'\leq vb$, $u'\equiv_{\bgq}u\equiv_{\bgq}a$, and $v'\equiv_{\bgq}v\equiv_{\bgq}b$.
Setting $y'\eqdef yu'$, we get
 \begin{gather*}
 \dd(y')=u'\dd(y)=u'\,,\\
 \rr(y')=yu'y^{-1}=v'\,,
 \end{gather*}
so $u'\scD_Sv'$.
Since $a\scD_Sb$, $u'\leq a$, and $v'\leq b$, and since~$S$ is $\scD$-cancellative, it follows that $a\sd u'\scD_Sb\sd v'$, that is, there is $s\in S$ such that $\dd(s)=a\sd u'$ and $\rr(s)=b\sd v'$.
{}From $u'\equiv_{\bgq}a$ it follows that $s\equiv_{\bgq}0$.
Set $x=y'\oplus s$.
Then $x\equiv_{\bgq}y'\equiv_{\bgq}y$, so $x\in\bx$.
Furthermore, $\dd(x)=a$ and $\rr(x)=b$.
\end{proof}

We can now state the promised projectivity statement for fully group-matricial biases.

\begin{lemma}\label{L:CancProj}
Let~$S$ be a $\scD$-cancellative Boolean inverse semigroup, let~$n$ be a positive integer, let~$H$ be a group, and let $\gf\colon S\twoheadrightarrow\Matp{n}{\pz{G}}$ be a surjective bias homomorphism.
Then there are a group~$\ol{G}$, a surjective group homomorphism $\psi\colon\ol{G}\twoheadrightarrow G$, and a bias embedding $\eta\colon\Matp{n}{\pz{\ol{G}}}\hookrightarrow S$ such that $\Matp{n}{\pz{\psi}}=\gf\circ\eta$.
\end{lemma}

The situation is illustrated on Figure~\ref{Fig:CancProj}.

\begin{figure}[htb]
\[
 \xymatrix{
 \Matp{n}{\pz{\ol{G}}}\ar@{->>}[rr]^{\Matp{n}{\pz{\psi}}}
 \ar@{^(->}[d]_{\eta} && \Matp{n}{\pz{G}}\\
 S\ar@{->>}[urr]_{\gf} &&
 }
 \]
\caption{A commutative triangle of biases}
\label{Fig:CancProj}
\end{figure}

\begin{proof}
For all $(i,j)\in[n]\times[n]$, denote by~$e_{i,j}$ the element of $\Matp{n}{\pz{G}}$ with $(i,j)$th entry equal to~$1$ and all other entries equal to~$0$.
By Lemma~\ref{L:InProj}, there is a system $(a_{i,j})_{(i,j)\in[n]\times[n]}$ of matrix units in~$S$ such that each $\gf(a_{i,j})=e_{i,j}$.
The subset
 \[
 \ol{G}=\setm{x\in S}{\dd(x)=\rr(x)=a_{1,1}}
 \]
is a subgroup of the monoid~$a_{1,1}Sa_{1,1}$.
For each $x\in\ol{G}$, $\dd(\gf(x))=\rr(\gf(x))=e_{1,1}$, thus $\gf(x)=\psi(x)_{(1,1)}$ for a unique $\psi(x)\in G$.
Clearly, $\psi$ is a group homomorphism from~$\ol{G}$ onto~$G$.

We claim that~$\psi$ is surjective.
Observe first that the kernel~$\bgf$ of~$\gf$ is an additive congruence of~$S$.
Let $g\in G$.
Since~$\gf$ is surjective, there exists $y\in S$ such that $\gf(y)=g_{(1,1)}$.
Set $\bx=y/{\bgf}$.
Since
 \[
 \gf(\dd(y))=\dd(\gf(y))=\dd(g_{(1,1)})=
 e_{1,1}=\gf(a_{1,1})\,,
 \]
we get $\dd(y)\equiv_{\bgf}a_{1,1}$, and similarly, $\rr(y)\equiv_{\bgf}a_{1,1}$.
Hence, $\dd(\bx)=\rr(\bx)=a_{1,1}/{\bgf}$.
By Lemma~\ref{L:LiftInv}, there exists $x\in\bx$ such that $\dd(x)=\rr(x)=a_{1,1}$; so $x\in\ol{G}$.
Moreover,
 \begin{align*}
 \gf(x)&=\gf(y)
 &&(\text{because }x\text{ belongs to }\bx=y/{\bgf})\\
 &=g_{(1,1)}\,, 
 \end{align*}
that is, $\psi(x)=g$, thus proving our claim.

For every $x\in\Matp{n}{\pz{\ol{G}}}$, it is not hard to verify that the elements $a_{i,1}x_{i,j}a_{1,j}$, for $(i,j)\in[n]\times[n]$, are pairwise orthogonal.
This enables us to set
 \[
 \eta(x)\eqdef\bigoplus_{(i,j)\in[n]\times[n]}
 a_{i,1}x_{i,j}a_{1,j}\,.
 \]
Elementary calculations show that~$\eta$ is an additive semigroup homomorphism, that is, a bias homomorphism, from $\Matp{n}{\pz{\ol{G}}}$ to~$S$.
Furthermore, for every $x\in\pz{\ol{G}}$, $\eta(x_{(1,1)})=a_{1,1}xa_{1,1}=x$, thus the restriction of~$\eta$ to the upper left corner of $\Matp{n}{\pz{\ol{G}}}$ is one-to-one.
By Lemma~\ref{L:CongMnS}, it follows that~$\eta$ is one-to-one.

Finally, for every $x\in\Matp{n}{\pz{\ol{G}}}$,
 \begin{align*}
 (\gf\circ\eta)(x)&=\bigoplus_{(i,j)\in[n]\times[n]}
 \gf(a_{i,1}x_{i,j}a_{1,j})\\
 &=\bigoplus_{(i,j)\in[n]\times[n]}
 e_{i,1}\psi(x_{i,j})_{(1,1)}e_{1,j}\\
 &=\bigoplus_{(i,j)\in[n]\times[n]}\psi(x_{i,j})_{(i,j)}\\
 &=\psi\pIII{\bigoplus_{(i,j)\in[n]\times[n]}(x_{i,j})_{(i,j)}}\\
 &=\psi(x)\,,
 \end{align*}
so $\gf\circ\eta=\psi$.
\end{proof}

\section{Boolean inverse semigroups with bounded index}
\label{S:BdedInd}

The main result of this section, Lemma~\ref{L:ind2eqn}, aims at relating the monoid-the\-o\-ret\-i\-cal
 concept of index, to be checked on the type monoid of a Boolean inverse semigroup~$S$, to the satisfaction of a certain inverse semigroup-theoretical identity, to be checked in~$S$.

\begin{lemma}\label{L:drxnn+1}
Let~$S$ be an inverse semigroup, let~$n$ be a positive integer, and let $x\in S$.
Then $\dd(x^n)=\rr(x^n)$ if{f} $\dd(x^n)=\dd(x^{n+1})$ and $\rr(x^n)=\rr(x^{n+1})$.
\end{lemma}

\begin{proof}(G. Kudryavtseva)
If $\dd(x^n)=\rr(x^n)$, then
 \[
 \dd(x^{2n})=x^{-n}x^{-n}x^nx^n=x^{-n}x^nx^{-n}x^n
 =(\dd(x^n))^2=\dd(x^n)\,,
 \]
whence $\dd(x^n)=\dd(x^{n+1})$.
The proof that $\rr(x^n)=\rr(x^{n+1})$ is similar.

Suppose, conversely, that $\dd(x^n)=\dd(x^{n+1})$ and $\rr(x^n)=\rr(x^{n+1})$.
For any positive integer~$k$, if $\dd(x^k)=\dd(x^{k+1})$, then
 \[
 \dd(x^{k+1})=x^{-1}\dd(x^k)x=x^{-1}\dd(x^{k+1})x
 =\dd(x^{k+2})\,.
 \]
Hence, our assumption implies that $\dd(x^n)=\dd(x^k)$ for every $k\geq n$.
In particular,
 \begin{equation}\label{Eq:ddrrxn2n}
 \dd(x^n)=\dd(x^{2n})\,. 
 \end{equation}
Now
 \begin{align*}
 x^{-n}x^n&=x^{-n}x^nx^{-n}x^n
 &&(\text{that is, }\dd(x^n)\text{ is idempotent})\\
 &=x^{-n}x^{2n}x^{-2n}x^n
 &&(\text{use \eqref{Eq:ddrrxn2n}})\\
 &=x^{-n}x^nx^nx^{-n}x^{-n}x^n\\
 &=x^{-n}x^nx^nx^{-n}
 &&(\text{the idempotents }x^{-n}x^n\text{ and }x^nx^{-n}
 \text{ commute})\\
 &\leq x^nx^{-n}\,. 
 \end{align*}
The proof that $x^nx^{-n}\leq x^{-n}x^n$ is symmetric.
%
\end{proof}

Recall the notation $\ext{x}{y}=xyx^{-1}$, used in~\cite{WBIS}.

\begin{lemma}\label{L:ConjugDecr}
Let~$S$ be a Boolean inverse semigroup and let $x,e\in S$ with~$e$ idempotent.
Then $\typ_S(\ext{x}{e})\lep\typ_S(e)$ within~$\Typ S$.
\end{lemma}

\begin{proof}
$\dd(xe)=\dd(x)e\leq e$, while $\rr(xe)=\ext{x}{e}$, so
$\typ_S(\ext{x}{e})=\typ_S(\rr(xe))=\typ_S(\dd(xe))\lep\typ_S(e)$.
\end{proof}

The identity $\dd(\vx^n)=\rr(\vx^n)$, the earliest appearance of which we are aware being~\cite[Theorem~3.4]{Reilly1980}, plays a crucial role in the following lemma.
It was suggested to the author by Ganna Kudryavtseva, together with a sketch of a proof of Corollary~\ref{C:ind2eqn}.
Our argument here is different.

\begin{lemma}\label{L:ind2eqn}
The following are equivalent, for any Boolean inverse semigroup~$S$ and every positive integer~$n$:
 \begin{enumeratei}
 \item\label{MonInd}
 $\typ_S(e)$ has index at most~$n$ in~$\Typ S$, for every $e\in\Idp S$;
 
 \item\label{drInd}
 $\dd(x^n)=\rr(x^n)$ for every $x\in S$.
 \end{enumeratei}
\end{lemma}

\begin{proof}
\eqref{MonInd}$\Rightarrow$\eqref{drInd}.
Set $e\eqdef\dd(x)\vee\rr(x)$ and $b=e\sd\rr(x)$.
Then $e=b\oplus\rr(x)=b\oplus\ext{x}{e}$, thus $\rr(x)=\ext{x}{e}=\ext{x}{b}\oplus\ext{x^2}{e}$, so $e=b\oplus\ext{x}{b}\oplus\ext{x^2}{e}$, and so on.
By an easy induction argument, we thus get
 \begin{equation}\label{Eq:CancDec}
 e=b\oplus\ext{x}{b}\oplus\cdots\oplus\ext{x^k}{b}
 \oplus\rr(x^{k+1})\,,\quad\text{for every }k\in\ZZ^+\,.
 \end{equation}
By Lemma~\ref{L:ConjugDecr}, it follows that
$(k+1)\cdot\typ_S(\ext{x^k}{b})+\typ_S(\rr(x^{k+1}))\lep\typ_S(e)$.
Hence, taking $k=n$ and by assumption on the index of $\typ_S(e)$, we get $\ext{x^n}{b}=0$.
By applying~\eqref{Eq:CancDec} to $k=n$ and $k=n+1$, we thus get
 \begin{align*}
 e&=b\oplus\ext{x}{b}\oplus\cdots\oplus\ext{x^{n-1}}{b}
 \oplus\rr(x^n)\\
 &=b\oplus\ext{x}{b}\oplus\cdots\oplus\ext{x^{n-1}}{b}
 \oplus\rr(x^{n+1})\,,
 \end{align*}
whence $\rr(x^n)=\rr(x^{n+1})$.
By applying that result to~$x^{-1}$, we get $\dd(x^n)=\dd(x^{n+1})$.
By Lemma~\ref{L:drxnn+1}, it follows that $\dd(x^n)=\rr(x^n)$.

\eqref{drInd}$\Rightarrow$\eqref{MonInd}.
Suppose that $\typ_S(e)$ has index greater than~$n$, where $e\in\Idp S$.
By the definition of the index and by \cite[Lemma~4-1.6]{WBIS}, there are nonzero pairwise orthogonal idempotents $e_0$, \dots, $e_n$ such that $\typ_S(e_0)=\typ_S(e_i)$ whenever $0\leq i\leq n$.
For $0\leq i<n$, there exists $x_i\in S$ such that $\dd(x_i)=e_{i+1}$ and $\rr(x_i)=e_i$.
Observe that $x_ix_j\neq 0$ if{f} $j=i+1$, whenever $0\leq i,j<n$.
Set $x=x_0\oplus\cdots\oplus x_{n-1}$.
Then $x^n=x_0\cdots x_{n-1}$, with $\dd(x^n)=e_n$ distinct from $\rr(x^n)=e_0$.
\end{proof}

\begin{corollary}[G. Kudryavtseva]\label{C:ind2eqn}
Let~$G$ be a group and let~$n$ and~$k$ be positive integers.
Then $\Matp{k}{\pz{G}}$ satisfies the identity $\dd(\vx^n)=\rr(\vx^n)$ if{f} $k\leq n$.
\end{corollary}

\begin{proof}
The idempotents of $\Matp{k}{\pz{G}}$ form a finite Boolean lattice with~$k$ pairwise $\scD$-equivalent atoms.
Hence $(\Typ\Matp{k}{\pz{G}},\typ(1))\cong(\ZZ^+,k)$.
Apply Lemma~\ref{L:ind2eqn}.
\end{proof}

\begin{definition}\label{D:Index}
Let~$S$ be a Boolean inverse semigroup.
\begin{itemize}
\item
Define the \emph{index of an element~$x$ of~$S$} as~$0$ if~$x=0$, the least positive integer~$n$ such that $\dd(x^n)=\rr(x^n)$ if it exists and $x\neq0$, and~$\infty$ in all other cases (i.e., $x\neq0$ and $\dd(x^n)\neq\rr(x^n)$ for every positive integer~$n$).

\item
Define the \emph{index of~$S$} as the supremum of all indexes of all elements of~$S$.

\end{itemize}
Moreover, define the \emph{index of a class~$\cC$ of Boolean inverse semigroups} as the supremum of all indexes of all members of~$\cC$.
\end{definition}

Now the following is a reformulation of Lemma~\ref{L:ind2eqn}.

\begin{corollary}\label{C:ind2eqn1}
Let~$S$ be a Boolean inverse semigroup.
Then the index of~$S$ is equal to the supremum of the indexes, within the type monoid~$\Typ S$, of all elements of the type interval $\Int S$.
\end{corollary}

\begin{corollary}\label{C:ind2eqn2}
Let~$\cV$ be a variety of biases.
Then the index of~$\cV$ is equal to the largest nonnegative integer~$n$ such that $\fI_n\in\cV$ if it exists, $\infty$ otherwise.
\end{corollary}

In particular, we emphasize that every variety of biases, distinct from the variety~$\Bis$ of all biases, has finite index (this follows from Theorem~\ref{T:VarMnG}).
Furthermore, for every positive integer~$n$, the class~$\Bis_n$ of all biases with index~$\leq n$ is a variety, defined by Reilly's identity $\dd(\vx^n)=\rr(\vx^n)$.
All its subdirectly irreducible members have the form $\Matp{k}{\pz{G}}$, where $0<k\leq n$ and~$G$ is a group.
In that sense, the identity $\dd(\vx^n)=\rr(\vx^n)$ is an analogue, for biases, of the Amitsur-Levitzki Theorem~\cite{AmLe50} for matrix rings.

\section{The variety order on fully group-matricial biases}
\label{S:VarOrder}

In this section we finally reach the main result of the paper, Theorem~\ref{T:VarOrdBias}, which states an isomorphism between proper varieties of biases and certain finite descending finite sequences of varieties of groups.

\begin{notation}\label{Not:nRad}
For a positive integer~$n$, the \emph{$n$-th radical}~$\Rad{n}{\cC}$ of a class~$\cC$ of biases is defined as the class of all groups~$G$ such that $\Matp{n}{\pz{G}}\in\cC$.
\end{notation}

\begin{lemma}\label{L:RadnVar}
Let~$\cV$ be a variety of biases.
Then $\Rad{n}{\cV}$ is either empty or a variety of groups, for every variety~$\cV$ of biases.
\end{lemma}

\begin{proof}
It is clear that $\Rad{n}{\cV}$ is closed under subgroups.
If a group~$H$ is a homomorphic image of a group~$G$, then, by Lemma~\ref{L:MnHom}, $\Matp{n}{\pz{H}}$ is a homomorphic image of~$\Matp{n}{\pz{G}}$; hence $G\in\Rad{n}{\cV}$ implies that $H\in\Rad{n}{\cV}$.
Finally, if~$I$ is a nonempty set and $\vecm{G_i}{i\in I}$ is a family of groups, then, setting $G=\prod_{i\in I}G_i$, the bias $\Matp{n}{\pz{G}}$ canonically embeds into $\prod_{i\in I}\Matp{n}{\pz{G_i}}$; hence $\setm{G_i}{i\in I}\subseteq\Rad{n}{\cV}$ implies that $G\in\Rad{n}{\cV}$.
The desired conclusion follows then from Birkhoff's HSP Theorem.
\end{proof}

By Corollary~\ref{C:ind2eqn2}, $\Rad{n}{\cV}$ is nonempty if{f}~$n$ is less than or equal to the index of the variety~$\cV$.
Also, observe that trivially, $\Rad{n+1}{\cV}\subseteq\Rad{n}{\cV}$.
The following lemma is the main technical result of this section.

\begin{lemma}\label{L:VarOrdbias}
Let~$I$ be a nonempty set, let $\vecm{n_i}{i\in I}$ be a bounded family of positive integers, let $\vecm{G_i}{i\in I}$ be a family of groups, let~$n$ be a positive integer, and let~$G$ be a group.
Then~$\Matp{n}{\pz{G}}$ belongs to the variety~$\cV$ of biases generated by $\setm{\Matp{n_i}{\pz{G_i}}}{i\in I}$ if{f} $n\leq n_o$ for some $o\in I$ and~$G$ belongs to the variety~$\cG$ of groups generated by $\setm{G_i\wr\fS_{\ip{n_i/n}}}{i\in I\,,\ n\leq n_i}$.
\end{lemma}

\begin{proof}
Suppose first that $n\leq n_o$ for some $o\in I$.
We must prove that $\Matp{n}{\pz{G}}$ belongs to~$\cV$, for each $G\in\cG$; that is, we must prove that~$\cG$ is contained in $\Rad{n}{\cV}$.

Let $i\in I$ with $n\leq n_i$.
It follows from Lemma~\ref{L:GBEmbed2} that $G_i\wr\fS_{\ip{n_i/n}}$ belongs to~$\Rad{n}{\cV}$.
By Lemma~\ref{L:RadnVar}, it follows that~$\cG$ is contained in~$\Rad{n}{\cV}$.

Suppose, conversely, that $\Matp{n}{\pz{G}}$ belongs to~$\cV$.
By Birkhoff's HSP Theorem, there are a Boolean inverse semigroup~$S$ and a surjective bias homomorphism $\gf\colon S\twoheadrightarrow\Matp{n}{\pz{G}}$ such that~$S$ embeds into a product of biases of the form $\Matp{n_i}{\pz{G_i}}$.
Setting $m=\max\setm{n_i}{i\in I}$, all biases $\Matp{n_i}{\pz{G_i}}$ belong to the variety~$\Bis_m$ of all biases of index at most~$m$, thus so does~$S$.
By Lemma~\ref{L:ind2eqn}, the index of~$\typ(e)$ in $\Typ S$ is at most~$m$, for every $e\in\Idp S$.
By \cite[Corollary~3.12]{WDim} (see also \cite[Lemma 1-5.3]{WBIS}), every element of~$\Typ S$ has finite index in~$\Typ S$.
By \cite[Proposition~3.13]{WDim} (see also \cite[Lemma 2-3.6]{WBIS}), the monoid $\Typ S$ is cancellative, thus (cf. Proposition~\ref{P:Dcanc}) $S$ is $\scD$-cancellative.
By Lemma~\ref{L:CancProj}, there are a group~$\ol{G}$, a surjective group homomorphism $\psi\colon\ol{G}\twoheadrightarrow G$, and a bias embedding $\eta\colon\Matp{n}{\pz{\ol{G}}}\hookrightarrow S$ such that $\Matp{n}{\pz{\psi}}=\gf\circ\eta$.
Since~$G$ belongs to the variety of groups generated by~$\ol{G}$, this reduces the problem to the case where $S=\Matp{n}{\pz{G}}$.

By possibly renaming the $(n_i,G_i)$, we can reduce the problem to the case where there is a bias embedding $\psi\colon\Matp{n}{\pz{G}}\hookrightarrow\prod_{i\in I}\Matp{n_i}{\pz{G_i}}$.
We may assume, in addition, that the $i$th component $\psi_i\colon\Matp{n}{\pz{G}}\to\Matp{n_i}{\pz{G_i}}$ of the map~$\psi$ is nonconstant, for every $i\in I$.
By Corollary~\ref{C:CongMnS}, the kernel~$\bgy_i$ of~$\psi_i$ has the form $\Matp{n}{\pz{\bgq_i}}$, where~$\bgq_i$ is the congruence of~$G$ associated to a normal subgroup~$H_i$ of~$G$.
Since~$\psi$ is one-to-one, the intersection of all congruences~$\bgq_i$ is the diagonal of~$G$, thus the intersection of all normal subgroups~$H_i$ is~$\set{1}$.
Now the bias homomorphisms $\Matp{n}{\pz{\gq_i}}\colon\Matp{n}{\pz{G}}\twoheadrightarrow\Matp{n}{\pz{(G/H_i)}}$ and $\psi_i\colon\Matp{n}{\pz{G}}\to\Matp{n_i}{\pz{G_i}}$ both have kernel $\Matp{n}{\pz{\bgq_i}}$.
Since $\Matp{n}{\pz{\gq_i}}$ is surjective, there is a unique bias embedding $\tau_i\colon\Matp{n}{\pz{(G/H_i)}}\hookrightarrow\Matp{n_i}{G_i}$ such that $\psi_i=\tau_i\circ\Matp{n}{\pz{\gq_i}}$.
By Lemma~\ref{L:GBEmbed2}, it follows that $n\leq n_i$ and $G/H_i$ embeds, as a subgroup, into $G_i\wr\fS_{\ip{n_i/n}}$.
Since~$G$ embeds into the product of all $G/H_i$ and each $G_i\wr\fS_{\ip{n_i/n}}$ belongs to~$\cG$, it follows that $G\in\cG$.
\end{proof}

Our next notation introduces an operator, denoted by~$\Wrn_{n}$, which sends any class of groups to either a variety of groups or the empty class.

\begin{notation}\label{Not:WRn}
For a class~$\cC$ of groups and a positive integer~$n$, we denote by $\Wr{n}{\cC}$ the variety of groups generated by $\setm{G\wr\fS_n}{G\in\cC}$ if $\cC\neq\es$, the empty class otherwise.
\end{notation}

\begin{lemma}\label{L:WRn}
Let~$\cC$ be a nonempty class of groups and let~$n$ be a positive integer.
Then $\Wr{n}{\cC}=\Wr{n}{\Var{\cC}}$.
\end{lemma}

\begin{proof}
The class of all groups~$G$ such that $G\wr\fS_n\in\Wr{n}{\cC}$ contains~$\cC$, and it is easily seen to be closed under subgroups, products, and homomorphic images.
By Birkhoff's HSP Theorem, it is thus a variety of groups.
Since it contains~$\cC$, it contains~$\Var\cC$.
\end{proof}

\begin{lemma}\label{L:WrCompAdd}
Let~$I$ be a nonempty set and let~$\vecm{\cG_i}{i\in I}$ be a family of group varieties.
Then $\Wr{n}{\bigvee_{i\in I}\cG_i}=\bigvee_{i\in I}\Wr{n}{\cG_i}$.
\pup{The join is evaluated within the lattice~$\bLg$ of all varieties of groups}.
\end{lemma}

\begin{proof}
Simply observe that $\bigvee_{i\in I}\cG_i=\Var\pI{\bigcup_{i\in I}\cG_i}$, and then use Lemma~\ref{L:WRn}.
\end{proof}

\begin{lemma}\label{L:RadDecrWr}
Let~$\cV$ be a variety of biases and let~$m$ and~$n$ be positive integers.
Then $\Wr{m}{\Rad{mn}{\cV}}$ is contained in $\Rad{n}{\cV}$.
\end{lemma}

\begin{proof}
If $mn$ is greater than the index of~$\cV$, then $\Rad{mn}{\cV}=\es$ and the result is trivial.
Suppose from now on that~$mn$ is less than or equal to the index of~$\cV$.
This ensures that $\Rad{mn}{\cV}$ is nonempty, so, by Lemma~\ref{L:RadnVar}, it is a variety of groups.
By Lemma~\ref{L:WRn}, it thus suffices to prove that $G\wr\fS_m$ belongs to $\Rad{n}{\cV}$, for each $G\in\Rad{mn}{\cV}$.
By Lemma~\ref{L:GBEmbed2}, $\Matp{n}{\pz{(G\wr\fS_m)}}$ embeds, as a bias, into $\Matp{mn}{\pz{G}}$, which, since $G\in\Rad{mn}{\cV}$, belongs to~$\cV$.
Hence, $G\wr\fS_m\in\Rad{n}{\cV}$, as required.
\end{proof}

\begin{notation}\label{Not:tbLb}
Denote by $\tbLg$ the set of all descending sequences $\vecm{\cG_n}{n\in\NN}$ of elements of $\bLg\cup\set{\es}$, such that only finitely~$\cG_n$ are nonempty, and $\Wr{m}{\cG_{mn}}$ is contained in $\cG_n$ for all positive integers~$m$ and~$n$.
We order~$\tbLg$ componentwise:
$\vecm{\cG_n}{n\in\NN}\leq\vecm{\cH_n}{n\in\NN}$ if $\cG_n\subseteq\cH_n$ for every $n\in\NN$.

For every $\cV\in\bLb\setminus\set{\Bis}$, we set
 \[
 \Rd{\cV}\eqdef\vecm{\Rad{n}{\cV}}{n\in\NN}\,.
 \]
Conversely, for every sequence $\vec{\cG}=\vecm{\cG_n}{n\in\NN}$ in $\tbLg$, we set
 \[
 \Mat{\vec{\cG}}\eqdef\bigvee_{n\in\NN}
 \Matp{n}{\pz{\cG_n}}\,,
 \]
where the join is evaluated within $\bLb\cup\set{\es}$, the empty join is by convention the trivial variety, and $\Matp{n}{\pz{\cG_n}}$ denotes the variety of biases generated by the class  $\setm{\Matp{n}{\pz{G}}}{G\in\cG_n}$ if $\cG_n\neq\es$, the empty class otherwise.
\end{notation}

A straightforward application of Lemma~\ref{L:WrCompAdd} yields the following.

\begin{lemma}\label{L:tbLgSubl}
The poset~$\tbLg$ is a sublattice of $(\bLg\cup\set{\es})^{\NN}$.
\end{lemma}

\begin{theorem}\label{T:VarOrdBias}
The assignments~$\Radn$ and~$\Matn$ define mutually inverse lattice isomorphisms between $\bLb\setminus\set{\Bis}$ and~$\tbLg$.
Consequently, $\bLb\cong\tbLg\sqcup\set{\infty}$.
\end{theorem}

\begin{proof}
Let~$\cV$ be a proper variety of biases.
Denote by~$h$ the index of~$\cV$ and set $\cG_n=\Rad{n}{\cV}$ for every $n\in\NN$.
It follows from the above, together with Lemma~\ref{L:RadDecrWr}, that the sequence $\Rd{\cV}=\vecm{\cG_n}{n\in\NN}$ belongs to~$\tbLg$.
It follows from the definition of~$\cG_n$ that $\Matp{n}{\pz{\cG_n}}$ is contained in~$\cV$, for every $n\in\NN$.
Conversely, since, by Theorem~\ref{T:VarMnG}, the variety~$\cV$ is generated by its fully group-matricial members, it is contained in the join of all $\Matp{n}{\pz{\cG_n}}$.
This proves that $\cV=\Mat{\Rd{\cV}}$.

Let, conversely, $\vec{\cG}=\vecm{\cG_n}{n\in\NN}\in\tbLg$\,, and denote by~$h$ the largest nonnegative integer such that $\cG_n\neq\es$ whenever $1\leq n\leq h$.
(The value $h=0$ is possible, in which case all $\cG_n=\es$.)
The class $\cV=\Mat{\vec{\cG}}$ is, by definition, a variety of biases.
Set $\cG'_n=\Rad{n}{\cV}$ for every $n\in\NN$.
By definition, $\cG_n\subseteq\cG'_n$.
Let, conversely, $G\in\cG'_n$; that is, $\Matp{n}{\pz{G}}$ belongs to the variety of biases generated by all $\Matp{k}{\pz{H}}$, where $1\leq k\leq h$ and $H\in\cG_k$.
By Lemma~\ref{L:VarOrdbias}, $G$ belongs to the variety of groups generated by all $H\wr\fS_{\ip{k/n}}$, where $n\leq k\leq h$ and $H\in\cG_k$.
Now for each such pair $(k,H)$,
 \begin{align*}
 H\wr\fS_{\ip{k/n}}&\in\Wr{\ip{k/n}}{\cG_k}
 &&(\text{by the definition of }\Wr{\ip{k/n}}{\cG_k})\\
 &\subseteq\Wr{\ip{k/n}}{\cG_{n\ip{k/n}}}
 &&(\text{because }n\ip{k/n}\leq k)\\
 &\subseteq\cG_{n}
 &&(\text{because }\vec{\cG}\in\tbLg)\,.
 \end{align*}
Hence, $G\in\cG_n$, which completes the proof that $\Rd{\Mat{\vec{\cG}}}=\vec{\cG}$.
\end{proof}

%

In particular, since~$\bLg$ is Arguesian (cf. Section~\ref{S:Basic}), so is~$\bLb$.
Similarly, $\bLb$ satisfies the duals of Haiman's identities from~\cite{Haim1985}.
However, all this is already a consequence of congruence-permutability, which we established, for biases, in \cite[Section~3-4]{WBIS}.
Since the problem whether every lattice identity, satisfied by all normal subgroup lattices of groups, also holds in the congruence lattice of every algebra in a congruence-permutable variety (or even in the congruence lattice of any loop), is still open, it is not clear at this point whether Theorem~\ref{T:VarOrdBias} would yield new identities valid in~$\bLb$.

\begin{corollary}\label{C:VarOrdBias2}
The assignment $\cG\mapsto\Matp{1}{\pz{\cG}}$ defines an isomorphism from~$\bLg$ onto a convex sublattice of~$\bLb$, with smallest element the variety of all idempotent biases \pup{i.e., generalized Boolean algebras}.
\end{corollary}

In particular, by Ol'\v{s}anski\u{\i}'s theorem quoted in Section~\ref{S:Basic}, $\bLb$ has cardinality the continuum.


\begin{thebibliography}{10}

\bibitem{AmLe50}
Avraham~Shimshon Amitsur and Jakob Levitzki, \emph{Minimal identities for
  algebras}, Proc. Amer. Math. Soc. \textbf{1} (1950), 449--463. \MR{0036751
  (12,155d)}

\bibitem{BuOW1994}
Robert~G. Burns and Sheila Oates-Williams, \emph{Varieties of groups and
  normal-subgroup lattices---a survey}, Algebra Universalis \textbf{32} (1994),
  no.~1, 145--152. \MR{1287020}

\bibitem{BurSan}
Stanley Burris and H.~P. Sankappanavar, \emph{{A Course in Universal Algebra}},
  Graduate Texts in Mathematics, vol.~78, Springer-Verlag, New York-Berlin,
  1981, out of print, Millenium Edition available online at
  \url{http://orion.math.iastate.edu/dpigozzi/BurrisSanka.pdf}. \MR{648287}

\bibitem{Freese76}
Ralph Freese, \emph{Planar sublattices of {${\rm FM}(4)$}}, Algebra Universalis
  \textbf{6} (1976), no.~1, 69--72. \MR{0398927 (53 \#2778)}

\bibitem{Freese1995}
\bysame, \emph{Alan {D}ay's early work: congruence identities}, Algebra
  Universalis \textbf{34} (1995), no.~1, 4--23. \MR{1344951 (97a:01050)}

\bibitem{GrUA}
George Gr{\"a}tzer, \emph{{Universal Algebra}}, second ed., Springer, New York,
  2008, With appendices by Gr{\"a}tzer, Bjarni J{\'o}nsson, Walter Taylor,
  Robert W. Quackenbush, G{\"u}nter H. Wenzel, and Gr{\"a}tzer and W. A. Lampe.
  \MR{2455216}

\bibitem{LTF}
\bysame, \emph{{Lattice Theory: Foundation}}, Birkh\"auser/Springer Basel AG,
  Basel, 2011. \MR{2768581 (2012f:06001)}

\bibitem{Haim1985}
Mark Haiman, \emph{Proof theory for linear lattices}, Adv. Math. \textbf{58}
  (1985), no.~3, 209--242. \MR{815357}

\bibitem{Howie}
John~M. Howie, \emph{{An Introduction to Semigroup Theory}}, Academic Press
  [Harcourt Brace Jovanovich, Publishers], London-New York, 1976, L.M.S.
  Monographs, No. 7. \MR{0466355 (57 \#6235)}

\bibitem{Huhn72}
Andr{\'a}s~P. Huhn, \emph{Schwach distributive {V}erb\"ande. {I}}, Acta Sci.
  Math. (Szeged) \textbf{33} (1972), 297--305. \MR{0337710 (49 \#2479)}

\bibitem{Jons53}
Bjarni J{\'o}nsson, \emph{On the representation of lattices}, Math. Scand
  \textbf{1} (1953), 193--206. \MR{0058567 (15,389d)}

\bibitem{Jons1972}
\bysame, \emph{The class of {A}rguesian lattices is self-dual}, Algebra
  Universalis \textbf{2} (1972), 396. \MR{0316325}

\bibitem{KLLR15}
Ganna Kudryavtseva, Mark~V. Lawson, Daniel~H. Lenz, and Pedro Resende,
  \emph{Invariant means on {B}oolean inverse monoids}, Semigroup Forum
  \textbf{92} (2016), no.~1, 77--101. \MR{3448402}

\bibitem{Laws98}
Mark~V. Lawson, \emph{{Inverse Semigroups}}, World Scientific Publishing Co.,
  Inc., River Edge, NJ, 1998, The theory of partial symmetries. \MR{1694900
  (2000g:20123)}

\bibitem{Laws10}
\bysame, \emph{A noncommutative generalization of {S}tone duality}, J. Aust.
  Math. Soc. \textbf{88} (2010), no.~3, 385--404. \MR{2827424 (2012h:20141)}

\bibitem{Laws12}
\bysame, \emph{Non-commutative {S}tone duality: inverse semigroups, topological
  groupoids and {$C\sp \ast$}-algebras}, Internat. J. Algebra Comput.
  \textbf{22} (2012), no.~6, 1250058, 47. \MR{2974110}

\bibitem{LaLe13}
Mark~V. Lawson and Daniel~H. Lenz, \emph{Pseudogroups and their \'etale
  groupoids}, Adv. Math. \textbf{244} (2013), 117--170. \MR{3077869}

\bibitem{LawSco14}
Mark~V. Lawson and Philip Scott, \emph{A{F} inverse monoids and the structure
  of countable {MV}-algebras}, J. Pure Appl. Algebra \textbf{221} (2017),
  no.~1, 45--74. \MR{3531463}

\bibitem{MMTa}
Ralph~N. McKenzie, George~F. McNulty, and Walter~F. Taylor, \emph{{Algebras,
  Lattices, Varieties. Vol. {I}}}, The Wadsworth \& Brooks/Cole Mathematics
  Series, Wadsworth \& Brooks/Cole Advanced Books \& Software, Monterey, CA,
  1987. \MR{883644 (88e:08001)}

\bibitem{McKin1943}
John C.~C. McKinsey, \emph{The decision problem for some classes of sentences
  without quantifiers}, J. Symbolic Logic \textbf{8} (1943), 61--76.
  \MR{0008991}

\bibitem{Munn1974}
Walter~D. Munn, \emph{Free inverse semigroups}, Proc. London Math. Soc. (3)
  \textbf{29} (1974), 385--404. \MR{0360881}

\bibitem{Neum1967}
Hanna Neumann, \emph{Varieties of {G}roups}, Springer-Verlag New York, Inc.,
  New York, 1967. \MR{0215899}

\bibitem{Olsh1970}
Alexander~Ju. Ol{\cprime}{\v{s}}anski{\u\i}, \emph{The finite basis problem for
  identities in groups}, Izv. Akad. Nauk SSSR Ser. Mat. \textbf{34} (1970),
  376--384. \MR{0286872}

\bibitem{Reilly1980}
Norman~R. Reilly, \emph{Varieties of completely semisimple inverse semigroups},
  J. Algebra \textbf{65} (1980), no.~2, 427--444. \MR{585734}

\bibitem{Tars49}
Alfred Tarski, \emph{{Cardinal Algebras. With an Appendix: Cardinal Products of
  Isomorphism Types, by Bjarni J{\'o}nsson and Alfred Tarski}}, Oxford
  University Press, New York, N. Y., 1949. \MR{0029954 (10,686f)}

\bibitem{Wallis2013}
Alistair~R. Wallis, \emph{Semigroup and category-theoretic approaches to
  partial symmetry}, Ph.D. thesis, Heriot-Watt University, Edinburgh, 2013.

\bibitem{WDim}
Friedrich Wehrung, \emph{The dimension monoid of a lattice}, Algebra
  Universalis \textbf{40} (1998), no.~3, 247--411. \MR{1668068 (2000i:06014)}

\bibitem{WBIS}
\bysame, \emph{{Refinement monoids, equidecomposability types, and Boolean
  inverse semigroups}}, hal-01197354, version 2, 216 p., 2016.

\end{thebibliography}

\providecommand{\noopsort}[1]{}\def\cprime{$'$}
  \def\polhk#1{\setbox0=\hbox{#1}{\ooalign{\hidewidth
  \lower1.5ex\hbox{`}\hidewidth\crcr\unhbox0}}}
  \providecommand{\bysame}{\leavevmode\hbox to3em{\hrulefill}\thinspace}
\providecommand{\MR}{\relax\ifhmode\unskip\space\fi MR }
\providecommand{\MRhref}[2]{%
  \href{http://www.ams.org/mathscinet-getitem?mr=#1}{#2}
}
\providecommand{\href}[2]{#2}

\end{document}